\documentclass[a4paper,12pt,reqno]{amsart}

\usepackage{amscd,amssymb,amsmath,amsfonts,amsthm, mathrsfs,xspace}
\usepackage{graphicx}
\usepackage{verbatim,enumerate,paralist}
\usepackage{textcomp}
\usepackage[pagebackref,colorlinks]{hyperref}
\hypersetup{%
  urlcolor=blue,linkcolor=blue,citecolor=blue
}
\usepackage{pdfpages}

\DeclareMathAlphabet{\mathbbe}{U}{bbold}{m}{n}

\renewcommand{\epsilon}{\varepsilon}
\renewcommand{\phi}{\varphi}

\usepackage[T1]{fontenc}
\usepackage[all,2cell,poly,knot,tips]{xy}
\UseAllTwocells

\newcounter{Definitioncount}
\newtheorem{theorem}{Theorem}[section] 
\newtheorem{lemma}[theorem]{Lemma}
\newtheorem{proposition}[theorem]{Proposition}
\newtheorem{corollary}[theorem]{Corollary}

\theoremstyle{definition}
\newtheorem{remark}[theorem]{Remark}
\newtheorem{example}[theorem]{Example}
\newtheorem{assumption}[theorem]{Assumption}

\newtheoremstyle{fact}{\bigskipamount}{\medskipamount}{\upshape}{}{\itshape}{. }{ }{Fact}
\theoremstyle{fact}

\newtheoremstyle{genquest}{\bigskipamount}{\medskipamount}{\upshape}{}{\itshape}{. }{ }{General Question}
\theoremstyle{genquest}

\newtheoremstyle{step}{2\bigskipamount}{\medskipamount}{\upshape}{}{\itshape}{. }{ }{\underline{Step~\thestep}}
\theoremstyle{step}

\renewcommand{\thestep}{\arabic{step}}

\newcommand{\dd}{\colon}
\newcommand{\lra}{\longrightarrow}
\newcommand{\lla}{\longleftarrow}
\newcommand{\ra}{\rightarrow}

\newcommand{\Lra}{\Longrightarrow}
\newcommand{\Lla}{\Longleftarrow}
\newcommand{\ldual}[1]{\mathord{{\let\nolimits\relax\sideset{^\wedge}{}{#1}}}}
\newcommand{\laction}[2]{\mathord{{\let\nolimits\relax\sideset{^{#1}}{}{#2}}}}
\newcommand{\conj}[2]{\mathord{{\let\nolimits\relax\sideset{^{#1}}{}{#2}}}}

\DeclareMathOperator{\Lan}{Lan}
\DeclareMathOperator{\Ran}{Ran}

\def\CA{{\mathscr A}}
\def\CB{{\mathscr B}}
\def\CC{{\mathscr C}}
\def\CD{{\mathscr D}}
\def\CE{{\mathscr E}}

\def\CI{{\mathscr I}}

\def\CK{{\mathscr K}}

\def\CM{{\mathscr M}}

\def\CP{{\mathscr P}}

\def\CR{{\mathscr R}}
\def\CS{{\mathscr S}}

\def\CV{{\mathscr V}}

\def\CX{{\mathscr X}}

\begin{document}

\title{Combinatorial categorical equivalences of Dold-Kan type}

\author{Stephen Lack}
\address{Department of Mathematics, Macquarie University, NSW 2109, Australia}
\email{steve.lack@mq.edu.au}

\author{Ross Street}
\address{Department of Mathematics, Macquarie University, NSW 2109, Australia}
\email{ross.street@mq.edu.au}

\thanks{Both authors gratefully acknowledge the support of the Australian Research Council Discovery Grant DP130101969; Lack acknowledges with equal gratitude the support of an Australian Research Council Future Fellowship}

\date{\today}
\maketitle

\begin{abstract}
We prove a class of equivalences of additive functor categories that are relevant to 
enumerative combinatorics, representation theory, and homotopy theory. 
Let $\CX$ denote an additive category with finite direct sums and split idempotents. 
The class includes (a) the Dold-Puppe-Kan theorem that simplicial objects in $\CX$ 
are equivalent to chain complexes in $\CX$; (b) the observation of Church, Ellenberg 
and Farb \cite{CEF2012} that $\CX$-valued species are equivalent 
to $\CX$-valued functors from the category of finite sets and injective partial functions; 
(c) a result T. Pirashvili calls of ``Dold-Kan type''; and so on. 
When $\CX$ is semi-abelian, we prove the adjunction that was an equivalence 
is now at least monadic, in the spirit of a theorem of D. Bourn.  
\end{abstract}\tableofcontents

\noindent {\small{\emph{2010 Mathematics Subject Classification:} 18E05; 20C30; 18A32; 18A25; 18G35; 18G30}

\section{Introduction}

The intention of this paper is to prove a class of equivalences of categories that 
seem of interest in enumerative combinatorics as per \cite{Species},
representation theory as per \cite{CEF2012}, 
and homotopy theory as per \cite{Berger1996}.
More specifically, for a suitable category $\CP$, we construct a category $\CD$
with zero morphisms (that is, $\CD$ has homs enriched in the category
$1/\mathrm{Set}$ of pointed sets) and an equivalence of categories of the form
\begin{equation}\label{basicequ}
[\CP,\CX] \simeq [\CD,\CX]_{\mathrm{pt}} \ .
\end{equation}
On the left-hand side we have the usual category of functors from
$\CP$ into any additive category $\CX$ which has finite direct sums
and split idempotents. 
On the right-hand side we have the category of functors which preserve
the zero morphisms.
We prove \eqref{basicequ} as Theorem~\ref{settingmainthm} below, under 
the Assumptions~\ref{ass1}-\ref{ass6} listed in Section~\ref{setting}.

One example has $\CP = \Delta_{\bot,\top}$, the category whose objects are 
finite ordinals with first and last element, and whose morphisms are functions
preserving order and first and last elements. Then $\CD$ is the category with
non-zero and non-identity morphisms
\begin{equation*}
0\stackrel{\partial}\lla 1 \stackrel{\partial}\lla 2\stackrel{\partial}\lla \dots 
\end{equation*}
such that $\partial \circ \partial = 0$.
Since there is an isomorphism of categories 
\begin{equation}\label{interval}
\Delta_{\bot,\top} \cong \Delta_+^{\mathrm{op}} \ ,
\end{equation} 
where the right-hand side is the algebraist's simplicial category (finite ordinals
and all order-preserving functions),
our result \eqref{basicequ} reproduces the Dold-Puppe-Kan Theorem
\cite{Dold, DoldPuppe, Kan}.  Our arguments equally apply to the full subcategory $\Delta_{\bot\neq\top}$ of $\Delta_{\bot,\top}$ obtained by removing the object $1$; functors $\Delta_{\bot\neq \top}\ra \CX$ are the traditional simplicial objects in $\CX$. 

Cubical sets also provide an example of our setting; see Example~\ref{Ex3+}.
We conclude that cubical abelian groups are equivalent to semi-simplicial
abelian groups. 

For a category $\CA$ equipped with a suitable factorization system 
$(\CE,\CM)$ \cite{FreydKelly},
write $\mathrm{Par}\CA$ (strictly the notation should also show the dependence 
on $\CM$) for the category with the same objects as $\CA$ and
with $\CM$-partial maps as morphisms. We identify $\CE$ with the subcategory of
$\CA$ having the same objects but only the morphisms in $\CE$.
Assume each object of $\CA$ has only finitely many $\CM$-subobjects.
Let $\CX$ be any additive category with finite direct sums and split idempotents.
Our main equivalence \eqref{basicequ} includes as a special case an equivalence of categories
\begin{equation}\label{parequiv}
[\mathrm{Par}\CA,\CX] \simeq [\CE,\CX] \ .
\end{equation}
(Here $\CD$ is obtained from $\CE$ by freely adjoining zero morphisms.)

Let $\mathfrak{S}$ be the groupoid of finite sets and bijective functions.
Let $\mathrm{FI}\sharp$ denote the category of finite sets and injective partial functions.
Let $\mathrm{Mod}^R$ denote the category of left modules over the ring $R$.
Our original motivation was to understand and generalize the classification theorem for 
$\mathrm{FI}\sharp$-modules appearing as Theorem 2.24 of \cite{CEF2012},
which provides an equivalence 
\begin{equation*}
[\mathrm{FI}\sharp, \mathrm{Mod}^R] \simeq [\mathfrak{S}, \mathrm{Mod}^R]
\end{equation*}
between the category of functors $\mathrm{FI}\sharp \ra \mathrm{Mod}^R$
and the category of functors $\mathfrak{S} \ra \mathrm{Mod}^R$. 
This is the special case of \eqref{parequiv} above in which $\CA$ is
the category $\mathrm{FI}$ of finite sets and injective functions, 
and $\CM$ consists of all the morphisms. 
For us, this result has provided a new viewpoint on representations of the 
symmetric groups, and a new viewpoint on Joyal species \cite{Species, AnalFunct}.

In order to consider stability properties of representations of the symmetric groups, 
the authors of \cite{CEF2012} also consider $\mathrm{FI}$-modules: 
that is, $R$-module-valued functors from the category $\mathrm{FI}$. 
Each $\mathrm{FI}\sharp$-module
clearly has an underlying $\mathrm{FI}$-module, so their Theorem 2.24
shows how symmetric group representations become $\mathrm{FI}$-modules.
Many $\mathrm{FI}$-modules naturally extend to $\mathrm{FI}\sharp$-modules;
in particular, projective $\mathrm{FI}$-modules do.
Inspired by results of Putman \cite{Put12}, Church et al.\cite{CEFN} have generalized their
stability results to $\mathrm{FI}$-modules over Noetherian rings and make it clear
how the stability phenomenon is about the $\mathrm{FI}$-module being finitely generated.      
One application they give is a structural version of the Murnaghan Theorem
\cite{Mur38, Mur55}, a problem which has its combinatorial aspects \cite{Sta99}. 
We are reminded of the way in which Mackey functors
\cite{Lindner1976, PanSt} give extra freedom to representation theory.

Another instance of an equivalence of the form \eqref{parequiv} is when
$\CA$ is the category of finite sets with its usual (surjective, injective)-factorization
system. Then $\mathrm{Par}\CA$ is equivalent to the 
category of pointed finite sets,
which is equivalent to Graeme Segal's category $\Gamma$ \cite{Segal}. 
After completing this work, we were alerted to Teimuraz Pirashvili's interesting 
paper \cite{Pirash00} which gives this finite sets example, 
makes the connection with Dold-Puppe-Kan, and discusses stable homotopy of 
$\Gamma$-spaces.  

We are very grateful to Aurelien Djament and Ricardo Andrade for 
alerting us to the
existence of Jolanta S\l omi\'nska's \cite{Slom} and Randall Helmstutler's \cite{Helm08}. 
These papers cover many of our examples and have many similar ideas.
In particular, a three-fold factorization is part of the setting of \cite{Slom};
while the additivity of $\CX$ is replaced in \cite{Helm08} (now revised and
published as \cite{Helm14}) by a semi-stable homotopy
model structure, and the equivalence by a Quillen equivalence.   
We believe our approach makes efficient use of established categorical techniques and is complementary to these other papers.
Of course, we should have discovered those papers ourselves.

Christine Vespa also emailed us suggesting connections with her joint paper \cite{HPV}.
We have added Example~\ref{ParExx}~(d) of type \eqref{parequiv} about the PROP for 
monoids as a step in that direction.

As a further example of type \eqref{parequiv}, we have Example~\ref{FLI(q)} about the
category of representations of the general linear groupoid $\mathrm{GL}(q)$ over the field
$\mathbb{F}_q$. 
A recent development in this area is the amazing equivalence 
$$[\mathrm{GL}(q), \mathrm{Mod}^R] \simeq [\mathrm{FL}(q), \mathrm{Mod}^R]$$
proved by Kuhn \cite{Kuhn14}; 
here $\mathrm{FL}(q)$ is the category of finite vector spaces over $\mathbb{F}_q$
and all linear functions, provided $q$ is invertible in $R$.
Another interesting development is the work of Putman-Sam \cite{PutSam} 
on stabilization results for representations of the finite linear groupoid $\mathrm{GL}(q)$,
for example, in place of $\mathfrak{S}$. Also see Sam-Snowden \cite{SamSno}.   

Consider the basic equivalence \eqref{basicequ}. It is really about Cauchy (or Morita)
equivalence of the free additive category on the ordinary category $\CP$
and the free additive category on the category with zero morphisms $\CD$. 
By the general theory of Cauchy completeness (see \cite{ECC} for example),
to have \eqref{basicequ} for all Cauchy complete additive categories $\CX$, it
suffices to have it when $\CX$ is the category of abelian groups. 
Cauchy completeness amounts to existence of absolute limits (see \cite{ACEC})
and, for additive categories,
amounts to the existence of finite direct sums and split idempotents. 

The strategy of proof of our main equivalence \eqref{basicequ} is as follows.
In Appendix~\ref{gr} we provide general criteria for deducing that a $\CV$-functor
between $\CV$-functor $\CV$-categories is an equivalence given that some other
$\CV$-functor between ``smaller'' $\CV$-functor $\CV$-categories is known to
be an equivalence. 
In Section~\ref{setting} we describe the setting for our main result and then,
in Section~\ref{basex} show that the partial map, Dold-Puppe-Kan, and 
cube examples satisfy our assumptions. 
We show in Section~\ref{ker} that the setting of Section~\ref{setting} 
leads to an adjunction between functor categories of a particularly combinatorial kind.
Section~\ref{reducto} reduces the proof that the adjunction is an equivalence to a
particular case, then Section~\ref{tpmc} treats that case.
Section~\ref{exx} provides some more examples.
In the spirit of Bourn's version \cite{Bourn07} of Dold-Kan, 
Section~\ref{semiabelian} shows, under an extra assumption on our basic setting, 
that if $\CX$ is semiabelian then, instead of
the equivalence \eqref{basicequ}, we have a monadic functor.
The extra assumption is related to Stanley \cite{Sta97} Theorem 3.9.2
as used in Hartl et al. \cite{HPV} Theorem 3.5.
We feel the extra assumption should not be necessary and that there should be a
proof for the semi-abelian case closer to that of the additive case but so far have 
been unable to find it.

\section{The setting}\label{setting}

Let $\CP$ be a category with a subcategory $\CM$ containing all the isomorphisms.

Assume we have a functor $(-)^*\dd \CM^{\mathrm{op}} \ra \CP$
which is the identity on objects and is 
such that $m^*\circ m=1$ for all $m\in \CM$.
In particular, the morphisms in $\CM$ are split monomorphisms (coretractions)
in $\CP$. 

We write $\mathrm{Sub}A$ for the (partially) ordered set of isomorphism 
classes of morphisms
$m\dd U\ra A$ into $A$ in $\CM$. We will use the term {\em subobject} 
rather than ``$\CM$-subobject'' for these elements. 
The order on $\mathrm{Sub}A$ is the usual one. 
Abusing notation for this order, we simply write $U\preceq V$ when there 
exists an $f \dd U\ra V$ such that $m=n\circ f$, where $m\dd U\ra A$ and $n \dd V\ra A$ in $\CM$. 
There can be confusion when $U=V$ as objects, so we assure the reader that we will take care.
A subobject of $A$ is {\em proper} when it is represented by a non-invertible $m\dd U\ra A$
in $\CM$; we write $U\prec A$ or $U\prec_m A$. 

Define $\CR$ to be the class of morphisms $r\in \CP$ with the property that, if $r = m\circ x \circ n^*$ with $m, n \in \CM$, then $m$ and $n$ are invertible. 

Define $\CD$ to be the category with zero morphisms (that is, $1/\mathrm{Set}$-enriched
category) obtained from $\CR$ by adjoining zero morphisms.
Composition in $\CD$ of morphisms in $\CR$ is as in $\CP$ if the result is itself
in $\CR$, but zero otherwise. 

Define $\CS$ to be the class of morphisms in $\CP$ of the form $r\circ m^*$ with
$m\in \CM$ and $r\in \CR$. 

\begin{assumption}\label{ass1}
Every morphism $f\in \CP$ factors as $f=n\circ r\circ m^*$ for $m,n\in \CM$ and $r\in \CR$, and these $m$, $n$, and $r$ are unique up to isomorphism.
\end{assumption}
\begin{assumption}\label{ass2}
If $r, r^{\prime}\in \CR$ are composable then $r^{\prime}\circ r \in \CS$. 
\end{assumption}
\begin{assumption}\label{ass3}
If $r,m^*\circ r \in \CR$ and $m\in \CM$ then $m$ is invertible.
\end{assumption}
\begin{assumption}\label{ass4}
The class $\CM \circ \CM^*$ of morphisms of the form $m\circ n^*$ with $m,n\in \CM$ is closed under composition.
\end{assumption}
\begin{assumption}\label{ass5}
For all objects $A\in \CP$, the set $\mathrm{Sub}A$ is finite. 
\end{assumption}
\begin{assumption}\label{ass6}
For all objects $A\in \CP$, the relation $R_A$ on the set $\mathrm{Sub}A$, defined
by $m R_A n$ if and only if $m^*\circ n\in \CM$, is contained in an antisymmetric
transitive relation on $\mathrm{Sub}A$. 
\end{assumption}

\begin{proposition}\label{SsinR}
If $t\circ s = m\circ r$ with $s,t\in \CS$, $r\in \CR$ and $m\in \CM$ then both $s$ and
$t$ are in $\CR$. 
\end{proposition}
\begin{proof}
First note that, by uniqueness in Assumption~\ref{ass1}, if $x\circ n^*=m'\circ r'$
in obvious notation, then $n^*$ is invertible. 
Now, with $s,t,m,r$ as in the Proposition, we can put $s= r_1\circ m_1^*$ and $t = r_2\circ m_2^*$. Then  $r_2\circ m_2^*\circ r_1\circ m_1^*= t\circ s = m\circ r$. So $m_1^*$ is
invertible and we conclude that $s\in \CR$. Using Assumption~\ref{ass1}, we have
$m_2^*\circ s = m_3\circ r_3\circ n_1^*$. Then $r_2\circ m_3\circ r_3\circ n_1^* = m\circ r$
implies $n_1^*$ invertible. So we may suppose $m_2^*\circ s = m_3\circ r_4$. 
Then $(m_2\circ m_3)^*\circ s = m_3^*\circ m_2^*\circ s = m_3^*\circ m_3\circ r_4 = r_4$.
By Assumption~\ref{ass3}, $m_2\circ m_3$ is invertible. It follows that $m_2$ is
invertible, so $t\in \CR$.
\end{proof}

Recall that the limit of a diagram consisting of a family of morphisms into
a fixed object $A$ is called a {\em wide pullback}; the morphisms in the
limit cone are called {\em projections}. 
The dual is {\em wide pushout}.  
The following result helps in checking examples.

\begin{proposition}\label{factorize}
Assume that the pullback of each morphism in $\CM$ along any morphism in $\CR$
exists and is in $\CM$.  
Assume wide pullbacks of families of morphisms in $\CM$ exist, 
have projections in $\CM$, and become wide pushouts under $m \mapsto m^*$.
Then Assumptions~\ref{ass1} and \ref{ass2} hold. 
\end{proposition}
\begin{proof}
For Assumption~\ref{ass1}, take $f \dd A \ra B$ in $\CP$.
Let $n\dd Y\ra B$ be the wide pullback of all those morphisms 
$V\ra B$ in $\CM$ through which $f$ factors.
Then $n\in \CM$ and there exists a unique $f_1$ with $f = n\circ f_1$.
Let $m\dd X\ra A$ be the wide pullback of those $p\dd U\ra A$ in $\CM$ 
such that the morphism $f_1$ factors through $p^*$.
Then $m\in \CM$ and $f_1 = r \circ m^*$ for a unique $r$.
Clearly $r\in \CR$ and we have uniqueness by a familiar argument.

Assumption~\ref{ass2} is proved as follows. Take $t\circ s$ with $s, t\in \CR$. 
We already know $t\circ s = m \circ u$ with $m\in \CM$ and $u\in \CS$.
Form the pullback $(n,a)$ of $(m,t)$. Then there exists $b$ with $a\circ b=u$,
$n\circ b = s$. From the latter, $n$ is invertible. 
So $t = m\circ a \circ n^{-1}$, and this implies $m$ invertible.
So $t\circ s \in \CS$.  
\end{proof}

While Assumption~\ref{ass6} is used only in the proof of
Theorem~\ref{specialthm}, it holds in all our examples because they satisfy the hypothesis of the following proposition.

\begin{proposition}\label{bbP}
Suppose each hom-set of $\CP$ is equipped with a reflexive, transitive, antisymmetric relation $\leq$ respected by composition on either side;
thus we have a locally posetal 2-category $\mathbb{P}$ with underlying category $\CP$. Suppose further that, for all $m\in \CM$, $m^*$ is right adjoint
to $m$ with identity unit. Then Assumption~\ref{ass6} is
redundant. Dually, the same is true if instead each $m^*$ is left
adjoint to $m$ with identity counit. 
\end{proposition}
\begin{proof} For $m$ and $n$ in $\CM$ with codomain $A$, 
let $m\trianglelefteq n$ mean that 
there exists $\ell$ in $\CM$ with $m\ell \leq n$ in $\mathbb{P}$.
We claim $\trianglelefteq$ is transitive, antisymmetric, and contains the relation $R_A$ of Assumption~\ref{ass6}. 

Suppose $m_1\trianglelefteq m_2\trianglelefteq m_3$. Then there are
$m$ and $n$ in $\CM$ with $m_1m\leq m_2$ and $m_2n\leq m_3$. So
$m_1mn\leq m_2n\leq m_3$ yielding transitivity.

Suppose $m_1\trianglelefteq m_2\trianglelefteq m_1$. Then there are $m$ and $n$ in $\CM$ with $m_1m\leq m_2$ and $m_2n\leq m_1$, and so $m_1mn\leq m_1$; but $m_1$ is fully faithful, so $mn\leq 1$. This gives a descending chain
$$\dots \leq (mn)^3\leq (mn)^2\leq (mn) \leq 1 $$
in $\CP(A,A)$. All terms of the chain are in the finite set $\mathrm{Sub}(A)$ so they cannot be distinct. So $(mn)^a = (mn)^b$ for some natural numbers $a > b$.
Since $mn$ is a monomorphism, $(mn)^{a-b} = 1$; so $m$ is a retraction
and a monomorphism, hence invertible. Thus $m_1=m_2$, proving $\trianglelefteq$ antisymmetric.

If $m R_A n$ then there is an $\ell \in \CM$ with $\ell = m^*\circ n$.
By the adjointness $m\dashv m^*$, it follows that $m\ell \leq n$.
So $m\trianglelefteq n$. This proves $R_A$ is contained in $\trianglelefteq$. 
\end{proof}

For each $u\dd A\ra B$ in $\CP$, there exist $m_u\in \CM$ and $s_u\in \CS$ as in the triangle
\begin{eqnarray}\label{fact_u}
\begin{aligned}
\xymatrix{
A \ar[rd]_{s_u}\ar[rr]^{u}   && B  \\
& S_u \ar[ru]_{m_u}  &
} 
\end{aligned}
\end{eqnarray}  
by Assumption~\ref{ass1}.

\begin{proposition}\label{esssub}
Suppose $u\dd A\ra B$ and $v\dd B\ra C$ in $\CP$. 
\begin{enumerate}
\item[(a)] If $s_{vu}\in \CR$ then $s_u\in \CR$. 
\item[(b)] If $s_{vu}\in \CR$ and $u\in \CS$ then $u, s_v\in \CR$.
\end{enumerate}
\end{proposition}
\begin{proof}
We prove (b) first. We have $v\circ u = m_v\circ s_v \circ u = m_v \circ m_{s_vu} \circ s_{s_vu}$. By uniqueness of factorization, we may take $s_{vu}= s_{s_vu}$.
So $s_v\circ u = m_{s_vu} \circ s_{vu}$. By Proposition~\ref{SsinR}, if $s_{vu}\in \CR$ and $u\in \CS$ then $s_v, u\in \CR$. 
Now we prove (a). We have $v\circ u = v \circ m_u\circ s_u$.
By (b), $s_{vm_us_u} = s_{vu} \in \CR$ implies $s_u\in \CR$.  \end{proof}

\section{Basic examples}\label{basex}

\begin{example}\label{Ex1+}
Take a category $\CA$ with a factorization system $(\CE,\CM)$ in the sense of \cite{FreydKelly}. 
Assume that the pullback of any morphism in $\CM$ along any morphism exists and is in $\CM$. 
Assume every morphism in $\CM$ is a monomorphism and every object
of $\CA$ has only finitely many $\CM$-subobjects. 
A span $f = (X\stackrel{f_0}\lla U \stackrel{f_1}\lra Y)$ is called a {\em partial map}
$f \dd X\lra Y$ when $f_0$ is in $\CM$.  There is a morphism from
such an $f$ to $g=(X\stackrel{g_0}\lla V \stackrel{g_1}\lra Y)$ just
when there is a (necessarily unique) morphism $U\to V$ making the two
evident triangles commute. 
Let $\CP= \mathrm{Par}\CA$ 
denote the category
whose objects are all those of $\CA$ and whose morphisms are isomorphism
classes $[f]$ of partial maps. This underlies a locally ordered
2-category with 2-cells as above. Composition is that of spans: that is, by pullback.
We identify $f\dd X \ra Y$ in $\CA$ with the morphism $[1_X,X,f] \dd X \lra Y$ in $\CP$.
In this way, we have the $\CM$ we require for $\CP$ as the one in $\CA$.
For $m\dd U\ra X$ in $\CM$, a right adjoint in $\mathrm{Par}\CA$ is defined by 
$m^*= [m,U,1_U] \dd X\ra U$, and clearly $m^*\circ m = 1$.
This gives our functor $(-)^*\dd \CM^{\mathrm{op}} \ra \CP$. 
Every partial map $f = (X\stackrel{f_0}\lla U \stackrel{f_1}\lra Y)$ has 
$$[f] = f_{1}\circ f_0^* $$
where $f_0\in \CM$. Furthermore, $f_1 = m\circ e$ uniquely up to isomorphism
for $m\in \CM$ and $e\in \CE$. 
It follows therefore that $\CR=\CE$ and that $[f]\in \CS$ if and only if $f_1\in \CE$.

Now we look at our Assumptions. For Assumption~\ref{ass1}, 
we make use of Proposition~\ref{factorize}. First note
that since we are assuming the $\CM$-subobjects form a finite set, 
finite wide pullbacks can be obtained from pullbacks.
By assumption, pullbacks of $\CM$s exist in $\CA$ and a pullback of an 
$\CM$ is an $\CM$ in a factorization system.
It is a pleasant exercise to see that these pullbacks remain pullbacks in $\CP$
and become pushouts in $\CP$ on taking left adjoints. 
Assumption~\ref{ass2} also follows from this but it is clear anyway since 
$\CR$ is closed under composition because $\CE$ is.

For Assumption~\ref{ass3}, take $r=[1_A,A,e]$ with $e\in \CE$ and 
$m=[1_V,V,m]$ in $\CM$. To have $m^*\circ r = [n,P,u] \in \CR$, we must have
$n$ invertible and $u\in \CE$. Then $e=m\circ u \circ n^{-1}$ implies $m\in \CE$;
so $m$ is invertible. 

The class of partial maps in Assumption~\ref{ass4} are those
of the form $[m,U,n]$ with $m,n\in \CM$; these are closed under composition since
pullbacks of $\CM$s exist and are in $\CM$. 

Assumption~\ref{ass5} was one of our assumptions on the factorization system on $\CA$. 

Notice that $\CS$ is not closed under composition unless each pullback of an $\CE$
along an $\CM$ is an $\CE$. This is true in many examples. 
\end{example}

\begin{example}\label{Ex2+}
Take $\CP$ to be the category $\Delta_{\bot,\top}$ of finite
non-empty ordinals $n = \{0,1,\dots,n-1\}$ with morphisms those functions which preserve
first element, last element and order. This underlies a locally
ordered 2-category in which there is a 2-cell from $\xi$ to $\zeta$
just when $\xi(i)\le \zeta(i)$ for each $i$ in the domain. 
Functors out of $\CP$ are augmented simplicial objects because of the isomorphism \eqref{interval}. 
Take $\CM$ to consist of all the injective functions in $\Delta_{\bot,\top}$;
each such injection $\partial$ has a left adjoint $\partial^*$ 
(and also a right adjoint for that matter) when regarded as a functor between ordered sets; 
clearly $\partial^*$ is surjective with $\partial^* \circ \partial = 1$ since $\partial$ is a fully faithful functor.
A surjection $\sigma$ in $\CP$ is of the form $\partial^*$ if and only if $\sigma(i)=0$
implies $i=0$.    
We write $\sigma_k \dd m+1 \ra m$ for the order-preserving surjection which takes the value $k$ twice.
We write $\partial_k \dd m \ra m+1$ for the order-preserving injection which does not have $k$ in its image. 
Note that $\partial_0\notin \CP$, while $\sigma_k \dashv \partial_k  \dashv \sigma_{k-1}$ 
as functors and $\sigma_k$ is a $\partial^*$ if and only if $k>0$.
Every $\xi \in \CP$ factors uniquely as 
$$\xi = \partial_{i_s}\circ \dots \circ \partial_{i_1}\circ \sigma_{j_1}\circ \dots \circ \sigma_{j_t}$$ 
for $0< i_1<\dots < i_s < n-1$ and $0\le j_1<\dots < j_t < m-1$. 

We claim $\CR$ consists of the identities and the surjections $\sigma_0 : m+1 \ra m$.
The only invertible morphisms in $\CP$ are identities.
Since members of $\CR$ factor through no proper injection, they must be surjective.
Every surjection $\tau$ is either of the form $\partial^*$ or uniquely of the form 
$\sigma_0\circ \partial^*$. 
Neither of these forms is permissible for $\tau \in \CR$ unless the injection $\partial$
is an identity. This proves our claim. 
It is also clear then that $\CS$ consists of all the surjections in $\CP$. 

To prove the Assumptions, we make use of the simplicial identities 
(see page 24 of \cite{GabZis} for example) which, 
apart from $\sigma_i\circ \partial_i = 1 = \sigma_{i-1}\circ \partial_i$,
say that, for all $i<j$, 
$$\partial_j\circ \partial_i = \partial_i \circ \partial_{j-1} \ , \ \sigma_{j-1}\circ \sigma_i = \sigma_i \circ \sigma_j \ , \ \sigma_{j}\circ \partial_i = \partial_i \circ \sigma_{j-1} \ , \ \sigma_{i}\circ \partial_{j+1} = \partial_j \circ \sigma_i \ .$$

The hypotheses of Proposition~\ref{factorize} hold since the pullback of any two 
monomorphisms in $\CP$ exists and is absolute (that is, preserved by all functors); 
see page 27 of \cite{GabZis} on the Eilenberg-Zilber Theorem. 
Alternatively, notice that the squares
\begin{eqnarray*}
\xymatrix{
n-1 \ar[rr]^-{\partial_{j-1}} \ar[d]_-{\partial_{i}} && n \ar[d]^-{\partial_{i}} \\
n \ar[rr]_-{\partial_{j}}  && n+1 }
\qquad
\xymatrix{
n-1  && n \ar[ll]_-{\sigma_{j-1}} \\
n \ar[u]^-{\sigma_i}  && n+1 \ar[u]_-{\sigma_i} \ar[ll]^-{\sigma_j}}
\end{eqnarray*}
are respectively a pullback and pushout in $\CP$ for $0< i < j < n$.
Then the general result follows by stacking these squares vertically and horizontally.

In the present example, $\CS$ is closed under composition, making Assumption~\ref{ass2} clear.

For Assumption~\ref{ass3}, suppose we have $\partial^* \circ \rho \in \CR$ and $\rho \in \CR$. Since $\rho$ is surjective, $\partial^* \circ \rho = 1$ implies $\rho = 1$ 
and hence $\partial = 1$, as required. Otherwise $\partial^* \circ \rho = \sigma_0$.
Yet, if $\rho = 1$ this contradicts the lack of right adjoint for $\sigma_0$.
So $\rho = \sigma_0$. Thus $\partial^* \circ \sigma_0 = \sigma_0$, and we can
cancel $\sigma_0$ to obtain again $\partial = 1$. 

For Assumption~\ref{ass4}, the class $\CM\circ \CM^*$ of morphisms consists of
those which reflect $0$. That is, $\xi = \mu \circ \partial^*$ with
$\mu\in \CM$ if and only if $\xi(i)=0$ implies $i=0$. This class is
clearly closed under composition. 

Finally, Assumption~\ref{ass5} is clear. 

Notice that the arguments above equally apply to the full subcategory  
$\Delta_{\bot\neq\top}$ of $\Delta_{\bot,\top}$ obtained by removing the object $1$.
Functors $\Delta_{\bot\neq \top}\ra \CX$ are the traditional simplicial objects in $\CX$. 
\end{example}

\begin{example}\label{Ex3+} This example is about the cubical category $\mathbb{I}$ as used by Sjoerd Crans \cite{CransThesis} and Dominic Verity \cite{LDTHOC, Verity07}. 
Functors with domain $\mathbb{I}$ are cubical objects in the codomain category. 
Verity constructed $\mathbb{I}$ as the free monoidal category containing a cointerval.

For each natural number $k$, define a poset $\langle k\rangle = \{-,1,2,\dots ,k,+\}$ 
by adjoining a bottom element $-$ and a top element $+$ to the discrete poset $\{1,2,\dots ,k\}$.
Any function $f \dd \langle k\rangle \ra \langle h\rangle$ which preserves top and bottom is order-preserving.
Thus we get a locally partially ordered 2-category with objects the $\langle k\rangle$, 
with morphisms the top-and-bottom-preserving functions, and with the pointwise order. 
Let $\mathbb{I}$ be the locally full sub-2-category
consisting of those $f\dd \langle k\rangle \ra \langle h\rangle$ for which,
if $f(i),f(j)\notin \{-,+\}$ then $i<j$ if and only if $f(i)<f(j)$. 

Let $\CP$ be the underlying category of this $\mathbb{I}$.
Let $\CM$ consist of the morphisms in $\CP$ which are injective as functions. 
Given such an $m \dd \langle k\rangle \ra \langle h\rangle$ in $\CM$,
define $m^* : \langle h\rangle \ra \langle k\rangle$ to send each $m(i)$
in the image of $m$ to $i$ and everything else to $+$. 
Clearly $m^*\in \CP$ and $m^*\circ m = 1$.
Furthermore, $mm^*(j)$ is equal to $j$ if $j=m(i)$ for some $i$, and $+$ otherwise.
Therefore $1\le m\circ m^*$ showing $m^*$ to be left adjoint to $m$
with identity counit.

We can characterize morphisms of the form $m^*$ as those which are surjective as functions and reflect the bottom element $-$.
Consequently $\CR$ consists of the morphisms which are surjective as
functions and reflect the top element $+$.

Assumption~\ref{ass5}~and~\ref{ass2} are clear. 
\begin{equation}\label{cubpbpo}
\begin{aligned}
\xymatrix{
\langle \ell \rangle \ar[rr]^-{q} \ar[d]_-{p} &&  \langle v \rangle \ar[d]^-{n} \\
\langle u \rangle \ar[rr]_-{m}  && \langle k \rangle}
\qquad
\xymatrix{
\langle k \rangle \ar[rr]^-{n^*} \ar[d]_-{m^*} && \langle v \rangle \ar[d]_-{q^*} \ar@/^/[ddr]^g \\
\langle u \rangle \ar[rr]^-{p^*} \ar@/_/[drrr]_{f} && \langle \ell \rangle \\
&& & \langle h \rangle }
\end{aligned}
\end{equation}

For Assumption~\ref{ass1}, again we use Proposition~\ref{factorize}.
The existence of intersections is obvious.
However, we must show that taking left adjoints gives cointersections.
Take a pullback as in the left-hand diagram of \eqref{cubpbpo} 
and consider the right-hand diagram.
Assume $f\circ m^*=g\circ n^*$.
It suffices to show $f\circ p\circ p^*=f$.
Now $fpp^*(i)$ is equal to $fp(j)$ if $i=p(j)$ for some $j$, and equal to $+$ otherwise.
In the first case, we have $fpp^*(i)=fp(j)=f(i)$, as required.
In the second case, if $i$ does not have the form $p(j)$ then
$m(i)$ is not in the image of $n$, and so $gn^* m(i)=+$;
thus $f(i)=fm^*m(i)=+$ and we again have $fpp^*(i)=f(i)$.

For Assumption~\ref{ass3}, suppose $m\in \CM$ and $r, m^*\circ r\in \CR$. 
Suppose $m^*(i)=+$. 
Since $r$ is surjective, we have $i=r(j)$ so that $m^*r(j)=+$.
However $m^*\circ r$ reflects $+$. So $j=+$, yielding $i=r(+)=+$.
This proves $m^*$ reflects $+$ and therefore must be invertible.

For Assumption~\ref{ass4}, the composites of the form $m\circ n^*$
are clearly the (not necessarily surjective) morphisms which reflect $-$.
Clearly these are closed under composition.

There is another possible characterization of $\CR$, namely as the category of right adjoints to the morphisms in $\CM$. 
Thus in fact $\CR$ is dual to $\CM$. 
Now $\CM$ is really just the category $\Delta_{\mathrm{inj}}$ of finite ordinals and injective order-preserving maps, 
and $\CR \cong \CM^{\mathrm{op}}$.
Also the category $\CM^*$, with morphisms the $m^*\in \CM$, is dual to $\CM$. 
The factorization of Assumption~\ref{ass1} in this case shows the category 
$\CP$ is a composite 
$$\mathbb{I} = \Delta_{\mathrm{inj}}\circ \Delta_{\mathrm{inj}}^{\mathrm{op}}\circ \Delta_{\mathrm{inj}}^{\mathrm{op}} \ ,$$
relative to suitably defined distributive laws.

An alternative viewpoint is that $\mathbb{I}$ is 
$\mathrm{Par}\mathrm{Par}\Delta_{\mathrm{inj}}$,
where in each case partial maps are defined relative to the morphisms in 
$\Delta_{\mathrm{inj}}$. 
\end{example}

\section{The kernel module}\label{ker}

We write $1/\mathrm{Set}$ for the category of pointed sets.
We write $X \wedge Y$ for the monoidal tensor (= smash product) on pointed sets.
Let $\mathrm{p}\Lambda$ denote the free pointed set 
$1+\Lambda$ on the set $\Lambda$.
This defines the value on objects of a strong monoidal functor 
$\mathrm{p} \dd \mathrm{Set} \ra 1/\mathrm{Set}$
whose right adjoint is the forgetful functor.

In the setting of Section~\ref{setting}, we shall define a functor
\begin{eqnarray}\label{moduleM}
M\dd \CD^{\mathrm{op}}\times \CP \lra 1/\mathrm{Set}
\end{eqnarray}
which preserves zeros in the first variable.

Using the notation of \eqref{fact_u}, define $M$ on objects by
\begin{eqnarray}\label{Mobj}
M(A,B) = \mathrm{p}\{u\in \CP(A,B) \dd s_u\in \CR \} \ .
\end{eqnarray}
Suppose $r\dd A_1 \ra A$ in $\CR$ and $f\dd B\ra B_1$ in $\CP$.
Define $M$ on morphisms by
\begin{eqnarray}\label{Mmor}
M(r,f)u = 
\begin{cases}
f\circ u \circ r & \text{for }  s_{fur}\in \CR \\
0 & \text{otherwise} \ .
\end{cases}
\end{eqnarray}
To verify that this is a functor we need to see that 
$$M(r_1,f_1)M(r,f)u = M(r\circ_{\CD} r_1,f_1\circ f)u$$ holds. 
This amounts to showing
that the left-hand side is non-zero if and only if the right-hand side is, 
because then both sides
equal $f_1\circ f\circ u \circ r\circ r_1$. 
In other words, we need to see that 
$s_{fur}, s_{f_1furr_1}\in \CR$ if and only if $r\circ r_1, s_{f_1furr_1}\in \CR$.
In fact, $s_{f_1furr_1}\in \CR$ implies both $s_{fur}\in \CR$ and $r\circ r_1\in \CR$.   
For we know by Assumption~\ref{ass2} that $r\circ r_1 \in \CS$; 
so, by Proposition~\ref{esssub}(b), $s_{f_1furr_1}\in \CR$ implies $r\circ r_1 \in \CR$.
By Proposition~\ref{esssub}(a), we also conclude that $s_{furr_1}\in \CR$;
then Proposition~\ref{esssub}(b) gives $s_{fur}\in \CR$.    

Let $\mathrm{p}_{*}\CP$ denote the free $1/\mathrm{Set}$-enriched category on the category $\CP$. 
For any locally pointed (that is, $1/\mathrm{Set}$-enriched) category $\CX$, we have an isomorphism of categories
$$[\mathrm{p}_{*}\CP,\CX]_{\mathrm{pt}} \cong [\CP,\CX] \ , $$ 
where, for emphasis, we write the subscript ``$\mathrm{pt}$'' 
for the pointed-set-enriched functor category.    

We identify \eqref{moduleM}
with its obvious extension to a $1/\mathrm{Set}$-functor
\begin{eqnarray}\label{additiveM}
M \dd \CD^{\mathrm{op}} \otimes \mathrm{p}_{*}\CP \lra 1/\mathrm{Set} \ .
\end{eqnarray}
Now we are in the situation for a kernel adjunction of the form \eqref{keradj} with 
$\CV = 1/\mathrm{Set}$. For any suitably complete and cocomplete 
$1/\mathrm{Set}$-category $\CX$, we have an adjunction
\begin{equation}\label{settingkeradj} 
\xymatrix @R-3mm {
[\CD,\CX]_{\mathrm{pt}} \ar@<1.5ex>[rr]^{\widehat{M}} \ar@{}[rr]|-{\bot} && [\CP,\CX]
 \ar@<1.5ex>[ll]^{\widetilde{M}} \ .
}
\end{equation}
More combinatorial expressions than \eqref{coendhat} and \eqref{endtilde} exist for $\widehat{M}$ and $\widetilde{M}$ in our present situation as we shall now see.
\begin{theorem}\label{combadj}
Suppose $M$ is the kernel as in \eqref{additiveM}.
Let $\CX$ be any suitably complete and cocomplete 
$1/\mathrm{Set}$-category. 
For any zero preserving functor $F\dd \CD\ra \CX$ and any functor 
$T\dd \CP \ra\CX$, there are isomorphisms
\begin{eqnarray*}
\widehat{M}(F)B \cong \sum_{S\preceq_m B}{FS} \ ,
\qquad
\widetilde{M}(T)A \cong \bigcap_{U\prec_m A} {\mathrm{ker} \ {T(m^*\dd A\ra U)}} \ .
\end{eqnarray*}
\end{theorem}
\begin{proof}
We shall prove $$\int^{A\in \CD}{M(A,B)\wedge FA} \cong \sum_{S\preceq_m B}{FS}$$
(where we are using $\Lambda\wedge X$ to denote the pointed-set-enriched tensor
of the pointed set $\Lambda$ with $X\in \CX$) by showing that the family of morphisms $\chi_A$ as defined by the diagram
\begin{equation}
\begin{aligned}
\xymatrix{
M(A,B)\wedge FA \ar[rr]^-{\chi_A} && \sum_{S\preceq_m B}{FS}  \\ 
FA \ar[u]^-{\mathrm{in}_u} \ar[rr]_-{Fs_u} && FS_u \ar[u]_-{\mathrm{in}_{m_u}} \ ,}
\end{aligned}
\end{equation}
for $u\in M(A,B)$, using the notation of \eqref{fact_u}, 
is a universal extraordinary natural transformation (that is, a universal wedge). 
In order for a family $\theta_A \dd M(A,B)\wedge FA \ra X$ in $\CX$ to be
extraordinary natural, we require 
\begin{eqnarray}\label{wedge}
\theta_A\circ \mathrm{in}_u\circ Fr = 
\begin{cases}
\theta_{A_1}\circ \mathrm{in}_{ur} & \text{for }  s_{ur}\in \CR \\
0 & \text{otherwise} \ .
\end{cases}
\end{eqnarray}
for all $r\dd A_1\ra A$ in $\CR$ and $u\dd A\ra B$ in $\CP$ with $s_u \in \CR$.
Notice that in this situation $s_{ur} = s_u\circ r$ by uniqueness of factorization and Assumption~\ref{ass2}. 

Now for the family $\chi$ we have $\chi_A\circ \mathrm{in}_u\circ Fr =  \mathrm{in}_{m_u}\circ F(s_u\circ_{\CD}r)$. 
If $s_{ur}\in \CR$ then $s_u\circ_{\CD}r = s_u\circ r = s_{ur}$, so
$\mathrm{in}_{m_u}\circ F(s_u\circ_{\CD}r) = \mathrm{in}_{m_u}\circ Fs_{ur} = \chi_{A_1} \circ \mathrm{in}_{ur}$; otherwise, $F(s_u\circ_{\CD}r) = F(0) = 0$, and we have our extraordinary naturality.

For universality, we must see that a general wedge $\theta$ factors uniquely through $\chi$.
Define $\phi \dd \sum_{S\preceq_m B}{FS} \ra X$ by $\phi \circ \mathrm{in}_m = \theta_S\circ \mathrm{in}_m$.
With $s_u\in \CR$, we have, using \eqref{wedge}, 
$$\theta_A \circ \mathrm{in}_u =  \theta_A \circ \mathrm{in}_{m_us_u} = 
\theta_{S_u}\circ \mathrm{in}_{m_u} \circ Fs_u = \phi \circ \mathrm{in}_{m_u} \circ Fs_u = \phi \circ \chi_A \circ \mathrm{in}_u  \ .$$
Thus $\theta =  \phi\circ\chi$. By taking $u=m\in \CM$, we also see that the definition of $\phi$ is forced. 

Define $\widetilde{T}A = \bigcap_{U\prec_m A} {\mathrm{ker} \ {T(m^*\dd A\ra U)}}$ 
with inclusion $\iota_A \dd \widetilde{T}A \ra TA$,
and, for $r\dd A\ra A_1$, define $\widetilde{T}r \dd \widetilde{T}A \ra \widetilde{T}A_1$ to be the restriction of $Tr$.
For this we need to see that $Tr \circ \iota_A$ factors through $\iota_{A_1}$.
Take a non-invertible $n\dd V\ra A_1$. Using Assumption~\ref{ass1}, we have $n^*\circ r = \ell \circ r_1 \circ m^*$. So $(n\circ \ell)^*\circ r = r_1\circ m^*$. If $m$ is invertible then $(n\circ \ell)^*\circ r \in \CR$. By Assumption~\ref{ass3}, $n\circ \ell$ is invertible. So $n$ is invertible, a contradiction. So $m$ is not invertible and we have 
$Tn^*\circ Tr \circ \iota_A = T(\ell \circ r_1) \circ Tm^* \circ \iota_A = T(\ell \circ r_1) \circ 0 = 0$, yielding $\widetilde{T}r$ with 
$\iota_{A_1} \circ \widetilde{T}r = Tr \circ \iota_A$.         

Put $\widehat{F}B = \sum_{S\preceq_m B}{FS}$ and transport the functoriality in $B$
across the isomorphism with $\widehat{M}(F)B$.
We will prove that there is a natural isomorphism
$$[\CP,\CX](\widehat{F}, T) \cong [\CD,\CX]_{\mathrm{pt}}(F,\widetilde{T})$$
and hence conclude that $\widetilde{M}(T) \cong \widetilde{T}$.
Take a natural transformation $\theta \dd \widehat{F}\Rightarrow T \dd \CP \ra \CX$. 
The definition of $\widehat{F}f$ involves writing $f\circ m = m_{fm} \circ s_{fm}$ for $m:U\to A$ and $m_{fm} : V\to B$ then, 
with $s_{fm}\in \CR$, we have commutativity in the following diagram. 
\begin{equation}\label{thetanat}
\begin{aligned}
\xymatrix{
FU \ar[r]^-{\mathrm{in}_U} \ar[d]_-{Fs_{fm}} & \widehat{F}A \ar[d]^-{\widehat{F}f} \ar[r]^-{\theta_A} & TA \ar[d]^-{Tf}\\
FV \ar[r]_-{\mathrm{in}_V} & \widehat{F}B \ar[r]_-{\theta_B} & TB}
\end{aligned}
\end{equation}
Consider a noninvertible $\ell \dd W \ra A$ in $\CM$. 
Then $\ell^* \in \CS$ and $\ell^* \notin \CR$.
So $T\ell^*\circ \theta_A \circ \mathrm{in}_{A} = \theta_W \circ \widehat{F}\ell^*\circ \mathrm{in}_{A} = \theta_W \circ 0 = 0$. This implies there exists a unique morphism $\phi_A \dd FA \lra \widetilde{T}A$ 
such that the following square commutes.
\begin{equation}\label{defphi}
\begin{aligned}
\xymatrix{
FA \ar[rr]^-{\phi_A} \ar[d]_-{\mathrm{in}_A} && \widetilde{T}A \ar[d]^-{i_A} \\
\widehat{F}A \ar[rr]_-{\theta_A} && TA}
\end{aligned}
\end{equation}
Naturality of $\phi$ is proved as follows using \eqref{thetanat} with $f=r \in \CR$:
$i_B\circ \phi_B\circ Fr = \theta_B \circ \mathrm{in}_{B} \circ Fr = \theta_B \circ \widehat{F}r \circ \mathrm{in}_{A} = Tr \circ \theta_A \circ \mathrm{in}_{A} = Tr \circ i_A \circ \phi_A = i_B\circ \widetilde{T}r\circ \phi_A $. 

For the inverse direction, take any natural transformation $\phi \dd F\Rightarrow \widetilde{T} \dd \CD\ra \CX$.
Define $\theta$ by commutativity of the following diagram.
\begin{equation}\label{thetadef}
\begin{aligned}
\xymatrix{
FU \ar[r]^-{\mathrm{in}_U} \ar[d]_-{\phi_U} & \widehat{F}A  \ar[r]^-{\theta_A} & TA \\
\widetilde{T}U \ar[rr]_-{i_U} &  & TU \ar[u]_-{Tm}}
\end{aligned}
\end{equation}       
We need to prove the right-hand square of \eqref{thetanat} commutes when precomposed
with any $\mathrm{in}_U$ for any $m:U\ra A$ in $\CM$.  
In the case where $s_{fm}\in \CR$, the desired commutativity is a consequence of the commutativity of the following three squares.
\begin{equation}\label{pfthetanat}
\begin{aligned}
\xymatrix{
FU \ar[r]^-{\phi_U} \ar[d]_-{Fs_{fm}} & \widetilde{T}U \ar[d]_-{\widetilde{T}s_{fm}} \ar[r]^-{i_A} & TU \ar[r]^-{Tm} \ar[d]^-{Ts_{fm}} & TA \ar[d]^-{Tf}\\
FV \ar[r]_-{\phi_V} & \widetilde{T}V  \ar[r]_-{i_B} & TV \ar[r]_-{Tn_{fm}} & TB}
\end{aligned}
\end{equation}
In the case where $s_{fm}\notin \CR$, we can write $s_{fm} = r\circ \ell^*$ for some noninvertible $\ell \in \CM$.
Then (using both $U\preceq A$ and $U\preceq U$) we have $Tf \circ \theta_A \circ \mathrm{in}_U = T(f\circ m)\circ i_U \circ \phi_U =T(n_{fm}\circ r) \circ T\ell^*\circ i_U \circ \phi_U = T(n_{fm}\circ r) \circ 0 \circ \phi_U = 0 = \theta_B \circ \widehat{F}f \circ \mathrm{in}_U$, as required.

To show that the assignments are mutually inverse, take $\theta$ and define $\phi$ by \eqref{defphi}. 
Let $\bar{\theta}$ be as $\theta$ is in \eqref{thetadef}. Then $\bar{\theta}_A\circ \mathrm{in}_U = Tm \circ i_U\circ \phi_U = Tm \circ \theta_U \circ \mathrm{in}_{U} = \theta_A \circ \widehat{F}m\circ \mathrm{in}_{U} = \theta_A \circ \mathrm{in}_U$. So $\bar{\theta} = \theta$. 

On the other hand, take $\phi$ and define $\theta$ by \eqref{thetadef}. 
Let $\phi^{\prime}$ be as $\phi$ is in \eqref{defphi}. 
Then $i_A \circ \phi^{\prime}_A = \theta_A \circ \mathrm{in}_{A} = i_A\circ \phi_A$. 
So $\phi^{\prime} = \phi$.    
\end{proof}

\section{Reduction of the problem}\label{reducto}

The problem being referred to in the title of this section is to show that the adjunction \eqref{settingkeradj} is an equivalence. 
We wish to use the theory in Appendix~\ref{gr} to reduce the problem to the 
special case where $\CR$ consists only of invertible morphisms. 
In Section~\ref{tpmc}, we will treat this special case. 

Consider $\CP$, $\CM$, $\CS$, $\CR$ and $\CD$ as in the setting of Section~\ref{setting}.

Let $\CI$ be the class of invertible morphisms in $\CP$.
Let $\CK = \CM\circ \CM^*$ be the subcategory of $\CP$ as assured by 
Assumption~\ref{ass4}.
Then $\CK$, $\CM$, $\CM^*$, $\CI$ and $\mathrm{p}_*\CI$ also fit the setting of Section~\ref{setting}.
Denote the kernel module for this situation by 
$N \dd \CI^{\mathrm{op}} \times \CK \ra 1/\mathrm{Set}$;
on objects it is given by $N(C,D) = \mathrm{p}\CM(C,D)$.  

Let $L\dd \CI \ra \CD$ and $K\dd \CK \ra \CP$ be the inclusions;
both are the identity on objects and so Cauchy dense. 

We have the obvious inclusion  
\begin{eqnarray}\label{specialtheta}
\theta_{C,D} \dd N(C,D) \ra M(LC,KD)
\end{eqnarray}
taking $m\in \CM(C,D)$ to $m \in \CP(C,D)$ (using the fact that $s_m = 1$).

\begin{proposition}\label{twocoends}
For $\theta \dd N \Rightarrow M(L,K)$ defined at \eqref{specialtheta}, 
the corresponding natural families
$$\theta^r_{A,D} \dd \int^{C\in \CI}{N(C,D)\wedge \CD(A,LC)} \lra M(A,KD) \ ,$$
$$\theta^{\ell}_{C,B} \dd \int^{D\in \CK}{\mathrm{p}\CP(KD,B)\wedge N(C,D)} \lra M(LC,B) \ ,$$
as in Appendix~\ref{gr}, are both invertible. 
\end{proposition}
\begin{proof}
For $r\dd A\ra C$ in $\CR$ and $m\dd C\ra D$ in $\CM$, 
we have $\theta^r_{A,D}[m\wedge r]= m\circ r$. 
Yet every $u\dd A\ra D$ with $s_u\in \CR$ factors uniquely up
to a morphism in $\CI$ as $u=m_u\circ s_u = \theta^r_{A,D}[m_u\wedge s_u]$.
So we have our inverse bijection.

By the definition of the second coend, for all $k\dd D\ra D_1$ in $\CK$ and 
$g\dd D_1\ra B$, we have $[(g\circ k)\wedge m]$ equal to $[g\wedge (k\circ m)]$
when $k\circ m \in \CM$, zero otherwise.
For $m\dd C\ra D$ in $\CM$ and $f\dd D\ra B$ in $\CP$, we have
$[f\wedge m] = [(f\circ m)\wedge 1] = [u\wedge 1]$ for some $u \in \CP$.
But $u = \ell \circ r \circ n^*$ for some $n,\ell \in \CM$ and $r\in \CR$.
So $[u\wedge 1] = [(\ell \circ r)\wedge n^*]$ provided $n^*\in \CM$, zero
otherwise. However, $n^*\in \CM$ implies $n$ invertible. 
So $s_u\in \CR$. In other words, $[f\wedge m]$ is either zero or of the
form $[u\wedge 1]$ for a unique $u$ with $s_u\in \CR$.  
By definition, $\theta^{\ell}_{C,B}[f\wedge m]= f\circ m$ when $s_{fm}\in \CR$,
zero otherwise. 
We have shown that an inverse to $\theta^{\ell}_{C,B}$ is provided by 
$u \mapsto [u\wedge 1_C]$.
\end{proof}

Combining Proposition~\ref{twocoends} and Corollary~\ref{gencor}, we have:
 
\begin{corollary}\label{NgivesM}
If $\widetilde{N}$ is an equivalence or crudely monadic then so is $\widetilde{M}$.
\end{corollary}

\section{The case $\CP = \CK$}\label{tpmc}

Let $\CK$ be a category with a subcategory $\CM$ which contains all the isomorphisms. 
Assume we have a functor $(-)^* \dd \CM^\mathrm{op} \ra \CK$ such that
$m^*\circ m = 1$ for all $m\in \CM$.
Assume every morphism $f$ in $\CK$ factors as $f = m\circ \ell^*$ with
$m, \ell \in \CM$. 
Then $(\CM^*,\CM)$ is a factorization system on $\CK$ in the sense of
\cite{FreydKelly}; indeed, $\CM^*$ consists of the retractions and $\CM$ of the coretractions.
We also assume the set $\mathrm{Sub}A$ of $\CM$-subobjects of each
object $A$ is finite. Let $\CI$ be the groupoid of invertible morphisms in $\CK$.
  
Let $J\dd \CI \ra \CM$ and $I \dd \CM \ra \CK$ be the inclusion functors. 
Both functors are the identity on objects.

Let $\CX$ be any category admitting finite coproducts and finite products.

\begin{lemma}\label{leftKan}
Each functor $F\dd \CI \ra \CX$ has a pointwise left Kan extension
\begin{equation}
\begin{aligned}
\xymatrix{
\CI \ar[rd]_{F}^(0.5){\phantom{a}}="1" \ar[rr]^{J}  && \CM \ar[ld]^{\mathrm{Lan}_JF}_(0.5){\phantom{a}}="2" \ar@{=>}"1";"2"^-{\kappa_F}
\\
& \CX 
}
\end{aligned}
\end{equation} 
along $J\dd \CI \ra \CM$ defined on objects by:
$$(\mathrm{Lan}_JF)X = \sum_{U\preceq_m X}FU \ .$$
For morphisms $f \dd X\ra Y$ in $\CM$, the following triangle commutes.
\begin{equation}
\begin{aligned}
\xymatrix{
FU \ar[drr]_-{\mathrm{in}_{fm}} \ar[rr]^-{\mathrm{in}_m} && (\mathrm{Lan}_JF)X \ar[d]^-{(\mathrm{Lan}_JF)f} \\
&& (\mathrm{Lan}_JF)Y}
\end{aligned}
\end{equation}
\end{lemma}
\begin{proof}
The inclusion $U\preceq_m X \mapsto (U, m \dd JU\ra X)$ 
of the discrete category on the set $\{U \dd U\preceq_m X\}$ into the comma category $J/X$ has a left adjoint, taking $(V,f\dd JV\ra X)$ to $V\preceq_f X$.
It is an adjoint equivalence. It follows that the colimit of 
$J/X \stackrel{\mathrm{dom}}\lra \CI \stackrel{F}\lra \CX$
can be calculated by restricting along the inclusion and so is the coproduct displayed.
\end{proof}

\begin{lemma}
Each functor $T\dd \CM \ra \CX$ has a pointwise right Kan extension
\begin{equation}
\begin{aligned}
\xymatrix{
\CM \ar[rd]_{T}^(0.5){\phantom{a}}="1" \ar[rr]^{I}  && \CK \ar[ld]^{\mathrm{Ran}_IT}_(0.5){\phantom{a}}="2" \ar@{<=}"1";"2"^-{\rho_T}
\\
& \CX 
}
\end{aligned}
\end{equation}  
along $I\dd \CM \ra \CK$ defined on objects by:
$$(\mathrm{Ran}_IT)X = \prod_{U\preceq_m X}TU \ .$$
For morphisms $k \dd X\ra Y$ in $\CK$, the square
\begin{equation}
\begin{aligned}
\xymatrix{
(\mathrm{Ran}_IT)X \ar[rr]^-{\mathrm{pr}_{m}} \ar[d]_-{(\mathrm{Ran}_IT)k} && TU \ar[d]^-{T\ell} \\
(\mathrm{Ran}_IT)Y \ar[rr]_-{\mathrm{pr}_n} && TV 
}
\end{aligned}
\end{equation}
commutes, where $n^*\circ k = \ell \circ m^*$ with $n, m, \ell \in \CM$.
\end{lemma}
\begin{proof}
The functor $U\preceq_m X \mapsto (m^* \dd X\ra IU, U)$ from the discrete category 
on the set $\{U \dd U\preceq_m X\}$ into the comma category
$X/I$ has a right adjoint, taking $(n\circ \ell^* \dd X\ra IA,A)$ to $W\preceq_n X$, and so is initial. 
It follows that the limit of 
$X/I \stackrel{\mathrm{cod}}\lra \CM \stackrel{T}\lra \CX$
can be calculated as the product displayed.
\end{proof}

This gives the two adjunctions
\begin{equation}\label{twoadj} 
\xymatrix @R-3mm {
[\CI,\CX] \ar@<1.5ex>[rr]^{\mathrm{Lan}_J} \ar@{}[rr]|-{\bot} && [\CM,\CX]
\ar@<1.5ex>[rr]^{\mathrm{Ran}_I}  \ar@{}[rr]|-{\top} \ar@<1.5ex>[ll]^{[J,1]} && [\CK,\CX]
\ar@<1.5ex>[ll]^{[I,1]}  \ .
}
\end{equation}

\begin{proposition}
The functor $[I,1] : [\CK,\CX]\ra [\CM,\CX]$ is comonadic.
\end{proposition}
\begin{proof}
The functor $I$ is Cauchy dense since it is bijective on objects.
The result follows by Proposition~\ref{Cdmonadicity}.
\end{proof}

\begin{proposition}
  If $\CX$ is an additive category with finite direct sums and
  split idempotents, then $\Lan_j\colon
  [\CI,\CX]\to[\CM,\CX]$ is comonadic.
\end{proposition}

\proof
Since the unit of the adjunction $\Lan_J\dashv [J,1]$ is a pointwise
coretraction $FA\to FA\oplus \bigoplus_{U<A} FU$, and hence a strong
monomorphism, the left adjoint $\Lan_J$ is conservative, so it will be
comonadic provided that it preserves certain equalizers. In fact, it preserves all finite limits. Since limits in $[\CM,\CX]$ are formed
pointwise, it suffices to see that each 
$$\mathrm{ev}_A\circ \mathrm{Lan}_J \dd [\CI,\CX]\lra \CX$$
preserves finite limits where $\mathrm{ev}_A \dd [\CM,\CX]\ra \CX$ is evaluation
at $A\in \CM$. By Lemma~\ref{leftKan}, 
$$\mathrm{ev}_A\circ \mathrm{Lan}_J \cong \bigoplus_{U\preceq A}\mathrm{ev}_U \ .$$
Each evaluation $\mathrm{ev}_U$ preserves limits and so their direct
sum does so. \endproof    

\begin{theorem}\label{specialthm}
For any additive category $\CX$
with finite direct sums and split idempotents,
the functor
\[ \hat{N}\colon [\CI,\CX] \lra [\CK,\CX] \]
is an equivalence of categories.
\end{theorem}

\proof
By the explicit formula for $\hat{N}$ of Theorem~\ref{combadj}, the
diagram
\[ \xymatrix{
    [\CI,\CX] \ar[rr]^{\hat{N}} \ar[dr]_{\Lan_J} && [\CK,\CX]
    \ar[dl]^{[I,1]} \\
    & [\CM,\CX] } \]
commutes up to isomorphism; thus since the downward arrows are
comonadic, $\hat{N}$ is induced by a morphism $\Theta\colon G\to H$ of
the comonads $G$ and $H$ induced by the adjunctions
$\Lan_J\dashv[J,1]$ and $[I,1]\dashv \Ran_I$,
respectively. Furthermore, $\hat{N}$ will be an equivalence if (and
only if) $\Theta$ is invertible. Now $\Theta$ is invertible just when
each component $\Theta_F\colon \Lan_J(FJ)\to \Ran_J(F)J$ is
invertible, and such a component is given by the composite
\[ \xymatrix{
    \hat{N}(FJ)I \ar[rr]^-{\eta'_{\hat{N}(FJ)}I} && \Ran_I(\hat{N}(FJ)I)I
    \ar[r]^-{\cong} & \Ran_I(\Lan_J(FJ))I \ar[rr]^-{\Ran_I(\epsilon)I} &&
    \Ran_I(F)I } \]
where $\eta'$ is the unit of the adjunction $[I,1]\dashv \Ran_I$ and
$\epsilon$ is the counit of the adjunction $\Lan_J\dashv [J,1]$. More
explicitly, this has component at $A$ given by a map
\[ \xymatrix{
    \sum\limits_{U\preceq_m A} FU\ar[rr]^-{(\Theta_F)_A} && \prod\limits_{V\preceq_n A} FV } \]
which can in turn be specified by giving a map $FU\to FV$ for each choice of
$m\colon U\preceq A$ and $n\colon V\preceq A$; and this is given by $F(n^*m)$
if $n^*m\in\CM$ and $0$ otherwise.

To conclude, we use Assumption~\ref{ass6} and Assumption~\ref{ass5} to deduce that each relation $R_A$ has a linear refinement on $\mathrm{Sub}(A)$.
So we can list all the subobjects as 
$$U_0\xrightarrow{1_A=m_0}A, U_1\xrightarrow{m_1}A, \dots , U_n\xrightarrow{m_n}A$$ 
such that $m_i R_A m_j$ implies $i \le j$.
Then the components of $\Theta$ are represented by an
upper-triangular matrix with identity morphisms in the main diagonal.
Since $\CX$ is additive (including existence of additive inverses), such
matrices are invertible.
\endproof 

\begin{remark} In case $\CX$ is a protomodular $1/\mathrm{Set}$-category 
(see Section~\ref{semiabelian} for references),
morphisms $f \dd A+B \ra A\times B$, with 
$\mathrm{pr}_1\circ f \circ \mathrm{in}_1 = 1_A$, 
$\mathrm{pr}_2\circ f \circ \mathrm{in}_2 = 1_B$ and $\mathrm{pr}_2\circ f \circ \mathrm{in}_1 = 0$, 
are strong epimorphisms. Even though $\Lan_J$ is no
longer comonadic, the comonad morphism
$\Theta\colon G\to H$ can still be defined, and it is a strong epimorphism.
We would like this to be the basis of a proof that $\widetilde{N}$ (and hence 
$\widetilde{M}$ using Corollary~\ref{NgivesM}) 
is crudely monadic but we do not know how. 
We will give a proof that $\widetilde{M}$ is crudely monadic in Section~\ref{semiabelian}
under the extra Assumption~\ref{ass7}.
\end{remark}

Combining Corollary~\ref{NgivesM} with Theorem~\ref{specialthm}, we obtain: 

\begin{theorem}\label{settingmainthm}
In the setting of Section~\ref{setting}, the adjunction \eqref{settingkeradj} is an equivalence for any additive category $\CX$ in which idempotents split and finite direct sums exist. 
\end{theorem}

\section{Examples of Theorem~\ref{settingmainthm}}\label{exx}

\begin{example}\label{babyDPK} 
We begin with a baby version of the Dold-Puppe-Kan Theorem.
Let $\mathrm{Pt}\CX$ denote the category whose objects are split epimorphisms in $\CX$, the morphisms are morphisms of the epimorphisms
which commute with the splittings; this is what Bourn \cite{Bourn91} calls the category of points in $\CX$.  
Take $\CP$ to be underlying category of the 2-category which is the free-living adjunction 
$\mu^*\dashv \mu$ with identity counit $\mu^*\circ \mu = 1$.
So $[\CP,\CX] \cong \mathrm{Pt}\CX$ and we have our $(-)^*$ functor choosing the left adjoint.  
Let $\CM$ consist of all the monomorphisms. 
Then $\CR$ contains only the identities. 
Theorem~\ref{settingmainthm} yields $$\mathrm{Pt}\CX \simeq \CX \times \CX \ .$$

This example also shows the necessity of $\CX$ having homs enriched in abelian groups (not merely commutative monoids). 
We need $\CX$ to have kernels of split epimorphisms and finite coproducts already. 
If we also ask that it have finite products then considering the split epimorphism $X\times Y\ra Y$ given by the projection, the counit 
is the canonical map $X+Y \ra X\times Y$, 
so if this is invertible we have hom enrichment in commutative monoids. 
Now considering the codiagonal $X+ X\ra X$, split by one of the 
injections, it is not hard to show that $1_X$ has an additive inverse. 
\end{example}

\begin{example}\label{DPK} (\cite{Dold, DoldPuppe, Kan}) Applying Theorem~\ref{settingmainthm} to 
Example~\ref{Ex2+} yields that $[\Delta_{\bot,\top},\CX]$ is equivalent to 
the category of non-negatively graded chain complexes in $\CX$.
\end{example}

\begin{example} 
Applying Theorem~\ref{settingmainthm} to Example~\ref{Ex3+} yields 
$[\mathbb{I},\CX] \simeq [\Delta_{\mathrm{inj}}^{\mathrm{op}},\CX]$, the category of semi-simplicial objects in $\CX$.
\end{example}
\bigskip

Here are some examples of the general type described in Example~\ref{Ex1+}.
\bigskip

\begin{example} A {\em (set) species} in the sense of Joyal \cite{Species} is a functor 
$F : \mathfrak{S} \lra \mathrm{Set}$
where $\mathfrak{S}$ is the category of sets and bijective functions. 
A {\em pointed-set species} is a functor 
$F : \mathfrak{S} \lra 1/\mathrm{Set}$.
An {\em $R$-module species} is a functor 
$F : \mathfrak{S} \lra \mathrm{Mod}^R$; the case where $R$ is
a field is the basic situation of \cite{AnalFunct}.

Following \cite{CEF2012}, we write $\mathrm{FI}$ for the category of finite sets
and injective functions. We then see that 
$\CM = \mathrm{FI}$ and $\CE = \mathfrak{S}$,
while $\CP = \mathrm{FI}\sharp$, the category of injective partial functions.  

\begin{corollary} \cite{CEF2012} The functor
$$\widehat{M} : [\mathfrak{S}, \mathrm{Mod}^R] \lra [\mathrm{FI}\sharp, \mathrm{Mod}^R]$$
is an equivalence of categories.
\end{corollary}
\end{example}

\begin{example}\label{FLI(q)}
Another example relevant to \cite{repfgl} is the category $\CA = \mathrm{FLI}(q)$
of finite vector spaces over the field $\mathbb{F}_q$ of cardinality 
$q$ (a prime power) and injective linear functions.
Let $\CE = \mathrm{GL}(q)$ be the category (groupoid) of finite
$\mathbb{F}_q$-vector spaces and bijective linear functions. 
Then $\CP = \mathrm{FLI}(q)\sharp$ is the
category of finite $\mathbb{F}_q$-vector spaces and injective 
partial linear functions.

\begin{corollary} The functor
$$\widehat{M} : [\mathrm{GL}(q), \mathrm{Mod}^R] \lra [\mathrm{FLI}(q)\sharp, \mathrm{Mod}^R]$$
is an equivalence of categories.
\end{corollary}
\end{example}   

\begin{example}\label{ParExx}
Here are a few examples of categories $\CA$ as in Example~\ref{Ex1+} to which 
Theorem~\ref{settingmainthm} applies with $\CE$ 
the surjections and $\CM$ the injections:
\begin{enumerate}[(a)]
\item the category of finite abelian groups and group morphisms;
\item the category of finite abelian $p$-groups and group morphisms;
\item the category of finite sets and all functions.
\item the category of finite sets and functions equipped with linear orders in their fibres
(this is the PROP for monoids).
\end{enumerate}
In example (d), not all pullbacks exist but all pullbacks along injective functions do. 
Theorem~\ref{settingmainthm} also applies to $\CM$ in place of $\CA$
in these examples. Then $\CE$ is replaced by the groupoid of
invertible morphisms in $\CA$. In case of example (a),
the paper \cite{HillarDarren2007} describes the groupoid being represented in $\CX$. 
\end{example}

\begin{example}
Consider the ``algebraic'' simplicial category $\Delta_+$ 
whose objects are all the natural numbers and 
whose morphisms $\xi : m\lra n$ are order-preserving functions 
$$\xi : \{0, 1, \dots , m-1\}\lra \{0, 1, \dots , n-1\} \ .$$
Put $\CA = \Delta_+^{\mathrm{op}}$.
Take $\CM$ in $\CA$ to consist of the surjections in $\Delta_+$.
Pushouts of surjections along arbitrary morphisms exist in 
$\Delta_+$. 
Then $\CE = \Delta_{+\mathrm{inj}}^{\mathrm{op}}$
and $\CP$ is the opposite of the category whose morphisms
$m\lra n$ are cospans $$m\stackrel{\xi}\lra r \stackrel{\sigma}\lla n$$
in $\Delta_+$ with $\sigma$ surjective. 
We could also take the ``topological'' simplicial category 
$\Delta$ (omit the object $0$) to obtain
a reinterpretation of the preoperads in $\CX$ in the sense of 
\cite{Berger1996}.
\end{example}

\begin{example}
Here is a rather trivial example involving $\Delta$.
Take $\CA$ to be the category of non-empty ordinals and morphisms
the order-preserving functions which preserve first element.
Let $\CM$ be the class of morphisms which are inclusions of
initial segments. This is part of a factorization system where
$\CE = \Delta_{\bot \neq \top}$ is the category of ordinals with 
distinct first and last element 
and morphisms the order-preserving functions which preserve first 
and last element. Sometimes $\CE$ is called the category of {\em intervals}; there is a duality isomorphism
$$\CE \cong \Delta^{\mathrm{op}} \ .$$ 
In this case, not only do we have the equivalence
$$[\CE, \CX] \simeq [\CP, \CX] $$
of Theorem~\ref{settingmainthm}, we actually also have an isomorphism
$$\CP \cong \CE \ .$$
\end{example}

\begin{example}
Take $\CA$ to be a (partially) ordered set with finite infima and the descending chain condition. 
Then every morphism is a monomorphism
and the strong epimorphisms are equalities. So $\CE$ is
the discrete category $\mathrm{ob}\CA$ on the set of
elements of the ordered set. The reader may like to contemplate
the case where $\CA$ is the set of strictly positive integers
ordered by division. 
\end{example} 

\section{When $\CX$ is semiabelian}\label{semiabelian}

Semiabelian categories include the category $\mathrm{Grp}$ of 
(not necessarily abelian) groups and group morphisms.
In \cite{Bourn07} Dominique Bourn gave a version of the Dold-Puppe-Kan Theorem (Example~\ref{DPK})
for the case where the codomain category $\CX$ was semiabelian.
In that case it asserted monadicity of the right adjoint in 
\begin{equation}\label{pttedkeradj} 
\xymatrix @R-3mm {
[\CD,\CX]_{\mathrm{pt}} \ar@<1.5ex>[rr]^{\widehat{M}} \ar@{}[rr]|-{\bot} && [\CP,\CX]
 \ar@<1.5ex>[ll]^{\widetilde{M}} \ .
}
\end{equation} 
In this section, we provide a version of this for $\CP$ as in Section~\ref{setting},
not only for $\Delta^{\mathrm{op}}$. 
However, we require an extra assumption related to idempotents.
Such conditions occur in Lawvere \cite{Law91} and Kudryavtseva-Mazorchuk \cite{KM09}.  

\begin{assumption}\label{ass7}
The maximal proper elements of $\mathrm{Sub}A$ can be listed $m_1, \dots, m_n$ 
such that the idempotents $c_i = m_i \circ m_i^*$ on $A$ satisfy 
$c_j\circ c_i\circ c_j=c_j\circ c_i$ for all $i< j$. 
\end{assumption}

Throughout we assume our category $\CX$ has zero morphisms (that is, has homs enriched in pointed sets).

We begin by providing a non-additive version of the material at the end of Appendix~\ref{idemp} on idempotents. 

\begin{proposition}
Suppose the category $\CX$ has kernels of idempotents.
Let $e,f$ be idempotents on an object $A$ of $\CX$.
If $e\circ f\circ e = e\circ f$ then the intersection of the kernels of $e$ and $f$ exists.
\end{proposition}
\begin{proof}
Let $k\colon K\ra A$ be the kernel of $e$.
Since $e\circ f \circ k=e\circ f\circ e\circ k =0$, there exists a unique $g$ with 
$f\circ k=k\circ g$. Then $k\circ g\circ g = f\circ k\circ g=f\circ f \circ k=f\circ k=k\circ g$ and $k$ is a monomorphism. So $g$ is idempotent. Then the kernel $\ell :L\ra K$ of $g$ is easily verified to be the intersection of the kernels of $e$ and $f$.
\end{proof}

Protomodular categories were defined by Bourn \cite{Bourn91}:
a category $\CX$ (with zero morphisms) is {\em protomodular} when 
it is finitely complete and, for each object $A$, 
the functor $\mathrm{ker}: \mathrm{Pt}A\ra \CX$ is conservative. 
Here $\mathrm{Pt}A$ is the category
whose objects $(p,X,s)$ consist of morphisms $p:X\ra A, s:A\ra X$ with
$p\circ s = 1_A$, and whose morphisms $f:(p,X,s)\ra (q,Y,t)$ are morphisms
$f:X\ra Y$ such that $q\circ f = p$ and $f\circ s = t$.
Also the functor $\mathrm{ker}$ takes $(p,X,s)$ to the kernel of $p$.

The following property is sometimes \cite{BorBou07} taken as the definition of protomodular
in the pointed context.  

\begin{lemma}\label{protolemma} 
In a protomodular category, if $(p,X,s)$ is an object of 
$\mathrm{Pt}A$ and $k:K\ra X$
is the kernel of $p$ then $s:A\ra X, k:K\ra X$ are jointly strongly epimorphic.
\end{lemma}
\begin{proof} Suppose $m:Y\ra X$ is a monomorphism and $m\circ u = s, m\circ v = k$ for some $u,v$. Then $m:(p\circ m,Y,u)\ra (p,X,s)$ is a morphism of
$\mathrm{Pt}A$. Using $v$, we see that $m$ induces an isomorphism between
the kernel of $p\circ m$ and $K$. Since $\mathrm{ker}: \mathrm{Pt}A\ra \CX$
is conservative, $m:(p\circ m,Y,u)\ra (p,X,s)$ is invertible. So $m$ is invertible.   
\end{proof}      

\begin{proposition}\label{protoidemp}
Let $a_1,\dots,a_n$ be a list of idempotents on an object $A$ of a protomodular category $\CX$. 
Suppose $a_i\circ a_j\circ a_i=a_i\circ a_j$ for $i<j$.
Suppose $a_i=m_i\circ m_i^*$ is a splitting of $a_i$ via a subobject $m_i:A_i\ra A$
and retraction $m_i^*$.
Let $k_i:K_i\ra A$ be the kernel of $m_i^*$ (or equally of $a_i$).
Then the morphisms $m_1,\dots , m_n$ along with the inclusion $\cap_i{\mathrm{ ker} \ m_i^*} \ra A$ are jointly strongly epimorphic.
\end{proposition}
\begin{proof}
By Lemma~\ref{protolemma}, for each $i$, the morphisms $m_i:A_i\ra A$
and $k_i:K_i\ra A$ are jointly strongly epimorphic; we will loosely say ``$A_i$ and $K_i$
cover $A$''. 

If $i>1$ then $a_1\circ a_i\circ k_1=a_1\circ a_i\circ a_1\circ k_1=0$, and so $a_i\circ k_1$
lands in $K_1$, providing a factorization $a_i\circ k_1=k_1\circ a_i^1$.
Now $a_i^1$ is also an idempotent, and, for $1<i<j$, $k_1\circ a_i^1\circ a_j^1\circ a_i^1= a_i\circ a_j\circ a_i\circ k_1 = a_i\circ a_j\circ k_1= k_1\circ a_i^1\circ a_j^1$, and so
$a^1_i\circ a^1_j\circ a^1_i= a^1_i\circ a^1_j$. Clearly the splitting $A^1_i$ of $a^1_i$
is contained in $A_i$.

The kernel $K^1_i$ of $a^1_i$ is $K^1_i=K_1\cap K_i$ since $a_1\circ x=a_i\circ x = 0$ is equivalent to
$x=k_1\circ y$ and $k_1\circ a^1_i\circ y = a_i\circ k_1\circ y = s_i\circ x = 0$; 
so in fact $a^1_i\circ y=0$. 

We know that $A$ may be covered by $A_1$ and $K_1$. By Lemma~\ref{protolemma} again, we know, for each $i>1$, that $K_1$ may be covered by the splitting of $A^1_i$
and the kernel $K^1_i=K_1\cap K_i$ of $a^1_i$. Since $A^1_i\le A_i$, we see that
$A$ may be covered by $A_1$, $A_i$, and $K_1\cap K_i$.

Now continue inductively.  
\end{proof}

\begin{theorem}\label{protocounit}
Suppose $\CP$ is as in Section~\ref{setting} and $\CX$ is protomodular (with zero morphisms) with finite coproducts.
Then the components of the counit of the adjunction \eqref{pttedkeradj} are strong epimorphisms.
\end{theorem}
\begin{proof}
The counit has components
$$\sum_{B\preceq_mA}{\bigcap_{C\prec_n B}{\mathrm{ker} \ Tn^*}} \lra TA \ .$$
We prove this is a strong epimorphism by induction on the number $k$ of maximal
proper subobjects $A_1, \dots, A_k$ of $A$, with $m_i:A_i\ra A$. 
The result is clear for $k=1$. For $k>1$, consider the following diagram.
\begin{eqnarray*}
\xymatrix{
\sum_i{\sum_{B\preceq_mA}{\bigcap_{C\prec_n B}{\mathrm{ker} \ Tn^*}}} \ar[rr]^-{ } \ar[d]_-{\delta} && \sum_i{TA_i} \ar[d]^-{\gamma} && \\
\sum_{B\prec_mA}{\bigcap_{C\prec_n B}{\mathrm{ker} \ Tn^*}} \ar[rr]_-{\alpha} && FA && \bigcap_{C\prec_n A}{{\mathrm{ker} \ Tn^*}} \ar[ll]^-{\beta}}
\end{eqnarray*}
The component of the counit is strongly epimorphic if and only if $\alpha$ and
$\beta$ are jointly strongly epimorphic.
The top row is a coproduct of components of the counit already known to be strongly
epimorphic by induction. 
So it suffices to show that $\gamma$ and $\beta$ are jointly strongly epimorphic.
Rewriting the domain of $\beta$ as
$$\bigcap_{C\prec_n A}{{\mathrm{ker} \ Tn^*}} = \bigcap_{i=1}^k{{\mathrm{ker} \ Tm_i^*}} \ ,$$
we see that Proposition~\ref{protoidemp} applies to yield what we want.
\end{proof}

A category $\CX$ is {\em semiabelian} \cite{JMT} when it has zero morphisms, 
is protomodular, is Barr exact, and has finite coproducts.
A category is {\em regular} \cite{Excat} when it is finitely complete, and has the 
(strong epimorphism, monomorphism)-factorization system existing 
and stable under pullbacks. 
It follows that every strong epimorphism is regular (that is, a coequalizer);
see \cite{CarSt86} for a proof. 
A category is {\em Barr exact} when it is regular and every equivalence relation 
is a kernel pair. 

We mentioned Bourn's category $\mathrm{Pt} \CX$ in Example~\ref{babyDPK}.
We will use the following routine fact.

\begin{lemma}\label{ptpresstepi}
If $\CX$ is a semiabelian category then the functor 
$\mathrm{Pt} \CX\ra \CX$, sending each split epimorphism to its kernel, 
preserves strong epimorphisms.
\end{lemma}
\begin{proof}
A strong epimorphism
\begin{eqnarray*}
\xymatrix{
X \ar[rr]^-{g} \ar@<1ex>[d]^{p} && Y \ar@<1ex>[d]^{q} \\
A\ar@<1ex>[u]^{s} \ar[rr]_-{f} && B\ar@<1ex>[u]^{t}}
\end{eqnarray*}
in $\mathrm{Pt} \CX$ has $f$ and $g$ strong epimorphisms in $\CX$.
From this it is easily verified that the square involving the downward-pointing
arrows is a pushout. 
Factor the morphism in $\mathrm{Pt} \CX$ as 
\begin{eqnarray*}
\xymatrix{
X \ar[rr]^-{g_1} \ar@<1ex>[d]^{p} && Z\ar[rr]^-{f_1} \ar@<1ex>[d]^{q_1} && Y \ar@<1ex>[d]^{q} \\
A\ar@<1ex>[u]^{s} \ar[rr]_-{1_A} && A\ar[rr]_-{f} \ar@<1ex>[u]^{t_1} && B\ar@<1ex>[u]^{t}
}
\end{eqnarray*}
in which the right-hand square involving the downward-pointing arrows is a pullback.
The induced morphism $\mathrm{ker} q_1\ra \mathrm{ker} q$ is invertible.
Semiabelian categories are Mal'tsev \cite{JMT}. 
Therefore, as a comparison morphism  to
the pullback in a pushout square of regular epimorphisms in a Mal'tsev exact category, $g_1$ is a strong epimorphism (see Theorem 5.7 of \cite{CKP93}). 
The induced morphism $\mathrm{ker} p\ra \mathrm{ker} q_1$ is the pullback
of $g_1$ along $\mathrm{ker} q_1 \ra Z$ and so is a strong epimorphism by regularity.
\end{proof}

\begin{theorem}
If $\CP$ is as in Section~\ref{setting} and $\CX$ is semiabelian then the functor
$\widetilde{M} \dd [\CP,\CX] \ra [\CD,\CX]_{\mathrm{pt}}$
of \eqref{settingkeradj} is (crudely) monadic.
\end{theorem}
 \begin{proof}
 By Theorem~\ref{protocounit}, the right adjoint $\widetilde{M}$ is conservative 
 (since this is logically equivalent to the counit being a strong epimorphism).
 Since $\CX$ is semiabelian, it has coequalizers. 
 Therefore $[\CP,\CX]$ has coequalizers, so, for crude monadicity \cite{CWM}, 
 it suffices to show that $\widetilde{M}$ preserves coequalizers of reflexive pairs.
 
 Both $[\CP,\CX]$ and $[\CD,\CX]_{\mathrm{pt}}$ are semiabelian.
 A limit-preserving functor between semiabelian categories preserves coequalizers of
 reflexive pairs provided it preserves strong (= regular) epimorphisms 
 (see Lemma 5.1.12 of \cite{BorBou04}).
 
 Let $q:S\ra T$ be a strong epimorphism in $[\CP,\CX]$.
 Each $q_A:SA\ra TA$ is a strong epimorphism.
 We must show each $\widetilde{q}_A:\widetilde{S}A\ra \widetilde{T}A$ is a strong epimorphism.
 
 Recall that $\widetilde{T}A$ is calculated by a sequence of kernels of split epimorphisms.
 This sequence depends only on the object $A$ and the category $\CP$, not on the
 particular functor $T$. 
 The desired result follows on repeated application of Lemma~\ref{ptpresstepi}.   
  \end{proof}
  
  \appendix
 \section{A general result from enriched category theory}\label{gr}

Let $\CV$ be a symmetric monoidal closed category which is complete and cocomplete. 
The tensor product of $A, B\in \CV$ is written $A\otimes B$ and the unit for
tensor is $I$, following \cite{KellyBook}.
For a $\CV$-category $\CX$, we also write $A\otimes X$ for the tensor
(=copower)  and
write $[A,X]$ for the cotensor  (=power)  of $A\in \CV$ and $X\in \CX$; 
we have
$$\CX(A\otimes X,Y)\cong \CV(A,\CX(X,Y)) \cong \CX(X,[A,Y]) \ .$$

A $\CV$-functor $L\dd \CC \ra \CA$ is {\em Cauchy dense} 
(as a morphism $L_*=\CA(1_{\CA},L)\dd \CC \ra \CA$
in the bicategory $\CV\text{-}\mathrm{Mod}$) when the morphism
\begin{eqnarray}\label{cd}
\int^C{\CA(LC,A_1)\otimes \CA(A,LC)}\ra \CA(A,A_1)
\end{eqnarray} 
induced by composition in $\CA$ is a strong epimorphism in $\CV$ 
for all $A, A_1\in \CA$. 
When $\CV$ is the category of abelian groups this means that each
$A\in \CA$ is a retract of a finite direct sum of objects $LC$ in the image of $L$. 
The following result is standard but we provide a proof.

\begin{proposition}\label{Cdmonadicity}
Suppose $\CX$ is a cocomplete $\CV$-category, $\CC$ is a small $\CV$-category,
and $L\dd \CC \ra \CA$ is a Cauchy dense $\CV$-functor. 
Then the $\CV$-functor $[L,1] \dd [\CA,\CX]\ra [\CC,\CX]$ is both monadic 
and comonadic, and so preserves and
reflects limits, colimits, and strong epimorphisms, as well as being conservative. 
\end{proposition}
\begin{proof}
The left adjoint to $[L,1]$ is left Kan extension $\mathrm{Lan}_L$ along $L$
defined by the coend formula
$$\mathrm{Lan}_L(G) = \int^C{\CA(LC,-)\otimes GC} \ .$$
Using the coend form of the Yoneda lemma (see \cite{efc}),
we see that the component at $F\in [\CA,\CX]$ of the counit for the adjunction
$\mathrm{Lan}_L \dashv [L,1]$ 
is isomorphic to the morphism
$$\int^{C,A}{(\CA(LC,-)\otimes \CA(A,LC))\otimes FA}\lra \int^A{\CA(A,-)\otimes FA}$$
induced by \eqref{cd}.
Since $L$ is Cauchy dense, these components are all strong epimorphisms.
It follows that the right adjoint $[L,1]$ is conservative.
It is clear that $[L,1]$ preserves colimits.
Since colimits exist in $[\CA,\CX]$, it follows that $[L,1]$ also reflects colimits.
By an easy case of the Beck monadicity theorem \cite{ML1963}, 
since $[L,1]$ has a left (respectively, right) adjoint,
is conservative and preserves coequalizers (respectively, equalizers), 
it is monadic (respectively, comonadic).    
\end{proof}

Suppose $M\dd \CB\ra \CA$ is a $\CV$-module viewed as a
$\CV$-functor $M\dd \CA^{\mathrm{op}} \otimes \CB \ra \CV$.
We call $M$ the {\em kernel} of the adjunction
\begin{equation}\label{keradj} 
\xymatrix @R-3mm {
[\CA,\CX] \ar@<1.5ex>[rr]^{\widehat{M}} \ar@{}[rr]|-{\bot} && [\CB,\CX]
 \ar@<1.5ex>[ll]^{\widetilde{M}} \ ,
}
\end{equation}
where 
\begin{eqnarray}\label{endtilde}
\widetilde{M}(T)A = \int_B{[M(A,B),TB]}
\end{eqnarray}
and
\begin{eqnarray}\label{coendhat}
\widehat{M}(F)B = \int^A{M(A,B)\otimes FA \ .}
\end{eqnarray}
The counit and unit of \eqref{keradj} will be denoted by
$$\varepsilon_M \dd \widehat{M}\widetilde{M}\Lra 1_{[\CB,\CX]} \ \text{ and } \ 
\eta_M \dd 1_{[\CA,\CX]} \Lra \widetilde{M} \widehat{M} \ .$$

To say $M\dd \CB\ra \CA$ is an equivalence as a $\CV$-module
is the same as saying $\widetilde{M}\dd [\CB,\CV]\ra [\CA,\CV]$
is an equivalence of $\CV$-categories. We also say $\CA$
and $\CB$ are {\em Cauchy} (or {\em Morita}) equivalent when
such an equivalence exists. 
It follows that $\CA^{\mathrm{op}}$
and $\CB^{\mathrm{op}}$ are also Cauchy equivalent,
and that the adjunction \eqref{keradj} is an equivalence for any
Cauchy complete $\CX$.   
See \cite{ECC} for more detail.    

Suppose $M\dd \CB\ra \CA$ and $N\dd \CD\ra \CC$ are $\CV$-modules, and 
$L\dd \CC \ra \CA$ and $K\dd \CD\ra \CB$ are $\CV$-functors.
Each $\CV$-natural family 
\begin{eqnarray}\label{theta}
\theta_{C,D}\dd N(C,D)\ra M(LC,KD)
\end{eqnarray}
induces $\CV$-natural families $\theta^{r}_{A,D}\dd$
$$\int^C{N(C,D)\otimes \CA(A,LC)} \stackrel{\int^C{\theta \otimes 1}}\lra \int^C{M(LC,KD)\otimes \CA(A,LC)} \stackrel{\mathrm{act^r}} \lra M(A,KD)$$ 
and $\theta^{\ell}_{C,B}\dd$
$$\int^D{\CB(KD,B)\otimes N(C,D)} \stackrel{\int^D{1 \otimes \theta}}\lra \int^D{\CB(KD,B)\otimes M(C,D)} \stackrel{\mathrm{act^{\ell}}} \lra M(LC,B) \ .$$
We obtain three $\CV$-natural transformations \eqref{threesquares}, all mates under
appropriate adjunctions.
\begin{eqnarray}\label{threesquares}
\begin{aligned}
\xymatrix{\ar @{} [dr] | {\stackrel{\theta^{r}} \Lla}
[\CA,\CX] \ar[d]_-{\widehat{M}} \ar[r]^-{[L,1]} & \ar @{} [dr] | {\stackrel{\theta^{\ell}} \Lra} [\CC,\CX] \ar[d]_-{\widehat{N}} \ar[r]^-{\mathrm{Lan}_L} & [\CA,\CX] \ar[d]^-{\widehat{M}}\\
[\CB,\CX] \ar[r]_-{[K,1]}           & [\CD,\CX]  \ar[r]_-{\mathrm{Lan}_K} & [\CB,\CX]
}
\qquad
\xymatrix{\ar @{} [dr] | {\stackrel{\tilde{\theta}} \Lra}
[\CB,\CX] \ar[d]_-{\widetilde{M}} \ar[r]^-{[K,1]} &  [\CD,\CX] \ar[d]_-{\widetilde{N}} \\
[\CA,\CX] \ar[r]_-{[L,1]}           & [\CC,\CX] 
}
\end{aligned}
\end{eqnarray}
Notice that $\tilde{\theta}$ is invertible if and only if $\theta^{\ell}$ is invertible.
Also, we have the following two commutative squares.
\begin{equation}\label{thetamates}
\begin{aligned}
\xymatrix{
\widehat{N} [L,1] \widetilde{M} \ar[rr]^-{\theta^r\widetilde{M}} \ar[d]_-{\widehat{N}\tilde{\theta}} && [K,1] \widehat{M}\widetilde{M} \ar[d]^-{[K,1]\varepsilon_M} \\
\widehat{N}\widetilde{N} [K,1] \ar[rr]_-{\varepsilon_N [K,1]} && [K,1]}
\qquad
\xymatrix{
[L,1] \ar[rr]^-{\eta_N[L,1]} \ar[d]_-{[L,1]\eta_M} && \widetilde{N}\widehat{N}[L,1] \ar[d]^-{\widetilde{N}\theta^r} \\
[L,1]\widetilde{M}\widehat{M} \ar[rr]_-{\tilde{\theta}\widehat{M}} && \widetilde{N}[K,1]\widehat{M}}
\end{aligned}
\end{equation}

\begin{proposition}\label{Mirror}
Suppose $\CX$ is a cocomplete $\CV$-category,
the $\CV$-functors $L\dd \CC \ra \CA$ and $K\dd \CD \ra \CB$ are Cauchy dense
with $\CC$ and $\CD$ small, and $\theta \dd N\Rightarrow M(L,K)$ is as in \eqref{theta}. 
\begin{enumerate}
\item[(a)] If $\tilde{\theta}$ is a strong epimorphism and $\widetilde{N}$ is conservative then
$\widetilde{M}$ is conservative. 
\item[(b)] If $\theta^r$ is a strong monomorphism and $\widehat{N}$ is conservative then $\widehat{M}$ is conservative.
\item[(c)] If $\tilde{\theta}$ is invertible and $\widetilde{N}$ preserves reflexive coequalizers then $\widetilde{M}$ preserves reflexive coequalizers.
\item[(d)] If $\theta^r$ is invertible and $\widehat{N}$ preserves reflexive coequalizers then $\widehat{M}$ preserves reflexive coequalizers.
\item[(e)] If $\tilde{\theta}$ and $\theta^r$ are invertible and $\widetilde{N}$ is fully faithful then $\widetilde{M}$ is fully faithful.
\item[(f)] If $\tilde{\theta}$ and $\theta^r$ are invertible and $\widehat{N}$ is fully faithful then $\widehat{M}$ is fully faithful.
\end{enumerate}
\end{proposition}
\begin{proof}
For (a), contemplate the left square of \eqref{thetamates}. 
If $\tilde{\theta}$ is a strong epimorphism, so too is $\widehat{N}\tilde{\theta}$.
If $\widetilde{N}$ is conservative then $\varepsilon_{N}[K,1]$ is a strong
epimorphism. It follows that $[K,1]\varepsilon_M$ is a strong epimorphism.
By Proposition~\ref{Cdmonadicity}, $\varepsilon_M$ is a strong epimorphism.
So $\widetilde{M}$ is conservative.  

For (b), contemplate the right square of \eqref{thetamates} and use the dual argument.

For (c), look at the rightmost square of \eqref{threesquares} with its invertible 2-cell. 
Take any kind of colimit in $[\CB,\CX]$. It is preserved by $[K,1]$ and so,
if the that kind of colimit is preserved by $\widetilde{N}$, then
$[L,1]\widetilde{M}$ preserves the colimit.
By Proposition~\ref{Cdmonadicity}, $[L,1]$ reflects the colimit and we have the result.  

For (d), look at the leftmost square of \eqref{threesquares} and apply the dual argument.

For (e), contemplate the left square of \eqref{thetamates} and deduce from our hypotheses that $[K,1]\varepsilon_M$ is invertible. By Proposition~\ref{Cdmonadicity}, $\varepsilon_M$ is invertible, so $\widetilde{M}$
is fully faithful.

For (f), contemplate the right square of \eqref{thetamates} and use the dual argument. 
\end{proof}

We shall say a $\CV$-functor is {\em crudely monadic} when it has a left adjoint,
is conservative, has reflexive coequalizers existing in the domain, and preserves
reflexive coequalizers.

\begin{corollary}\label{gencor}
Suppose $L$ and $K$ are Cauchy dense and both $\theta^{\ell}$ and $\theta^{r}$
are invertible. 
\begin{enumerate}
\item[(a)] If $\widetilde{N}\dd [\CD,\CV]\ra [\CC,\CV]$ is crudely monadic then $\widetilde{M}\dd [\CB,\CV]\ra [\CA,\CV]$ is too. 
\item[(b)] If $\widetilde{N}\dd [\CD,\CV]\ra [\CC,\CV]$ is an equivalence then $\widetilde{M}\dd [\CB,\CV]\ra [\CA,\CV]$ is too. 
\end{enumerate}
\end{corollary}

\section{Remarks on idempotents}\label{idemp}

Define a relation on any monoid $M$ by $a\sqsubseteq b$ when $ba = a$.
This relation is transitive; indeed, we have a stronger property.

\begin{proposition}\label{gentrans}
If $ub\sqsubseteq a$ and $vc \sqsubseteq b$ then $uvc\sqsubseteq a$.

If $u_ia_i\sqsubseteq a_{i-1}$ for $1\sqsubseteq i\sqsubseteq n$ then $u_1\dots u_na_n\sqsubseteq a_0$.
\end{proposition}
\begin{proof}
For the first sentence, the assumptions are $aub = ub$ and $bvc=vc$.
So $auvc = aubvc=ubvc=uvc$; that is, $uvc \sqsubseteq a$. 
The second sentence follows by induction.
\end{proof}

The relation is only reflexive for idempotents: clearly $a\sqsubseteq a$ is equivalent to $aa=a$.

The unit $1$ of the monoid is a largest element in the sense that $a\sqsubseteq 1$ for all $a\in M$.
That is, $1$ is the empty meet. 
However, not all meets need exist. A {\em meet} for $a,b\in M$ is an element $a\wedge b$
with $a\wedge b \sqsubseteq a$ and $a\wedge b \sqsubseteq b$, and, if $x\sqsubseteq a$ and $x\sqsubseteq b$ then
$x\sqsubseteq a\wedge b$. In particular, $a\wedge b$ must be an idempotent. 
Meets of lists of $n$ elements are defined in the obvious way and, for $n\ge 2$,
can be constructed from iterated binary meets when they exist. 

\begin{proposition}\label{someets}
If $ab\sqsubseteq b$ and $a\sqsubseteq a$ then $a\wedge b = ab$.

If $a_1, \dots , a_n$ are idempotents such that $a_ia_j\sqsubseteq a_j$ for $i\sqsubseteq j$
then $a_1\wedge \dots \wedge a_n = a_1 \dots a_n$. 
\end{proposition}
\begin{proof}
For the first sentence, we are told that $ab\sqsubseteq b$, while $a\sqsubseteq a$ implies $aa=a$,
and so $aaab=ab$, yielding $ab\sqsubseteq a$.
For the second sentence the result is clear for $n=1$ since we suppose $a_1$
idempotent. Assume the result for $n-1$; so $a_1\wedge \dots \wedge a_{n-1} = a_1 \dots a_{n-1}$. Apply the second sentence of Proposition~\ref{gentrans} to the inequalities $a_ia_n\sqsubseteq a_n$ to deduce $a_1\dots a_{n-1}a_n\sqsubseteq a_n$.
So, by the first sentence, $a_1\dots a_n = a_1\dots a_{n-1}\wedge a_n = a_1\wedge \dots \wedge a_{n-1}\wedge a_n$, as required.
\end{proof}

Notice that, if $ab=ba$ and $b$ is idempotent, then $bab=abb=ab$, so $ab\sqsubseteq b$.
So the proposition applies to commuting idempotents.

Now suppose we have a ring $R$. We can apply our results to the multiplicative monoid of $R$. 
We say idempotents $e$ and $f$ in $R$ are {\em orthogonal}
when $ef=fe=0$. A list $e_0, e_1, \dots e_n$ of idempotents is {\em orthogonal}
when each pair in the list is orthogonal. The list is {\em complete} when
$e_0+ e_1+ \dots + e_n = 1$. An easy induction shows that a complete
list of idempotents is orthogonal if and only if $e_ie_j = 0$ for $i\le j$.
 
For each $a\in R$, put $\bar{a}=1-a$. Clearly if $a$ is idempotent, so is $\bar{a}$.

Let $R^{\circ}$ denote the ring obtained from $R$ by reversing multiplication.

\begin{proposition}
\begin{enumerate}
\item[(a)] $a\sqsubseteq b$ in $R$ if and only if $\bar{b}\sqsubseteq \bar{a}$ in $R^{\circ}$. 
\item[(b)] For $b$ an idempotent, 
$ab\sqsubseteq b$ in $R$ if and only if $\bar{a}\bar{b}\sqsubseteq \bar{b}$ in $R^{\circ}$.
\item[(c)] If $a$ and $b$ are idempotents and $ab\sqsubseteq b$ in $R$ then 
$e_0=ab, e_1=\bar{a}b, e_2=\bar{b}$
is a complete list of orthogonal idempotents. 
\end{enumerate}
\end{proposition}
\begin{proof}
\begin{enumerate}
\item[(a)] $\bar{b}\sqsubseteq \bar{a}$ in $R^{\circ}$ means $(1-b)(1-a) = 1-b$ in $R$; 
that is, $ba=a$ which means $a\sqsubseteq b$ in $R$. 
\item[(b)] $\bar{a}\bar{b}\sqsubseteq \bar{b}$ in $R^{\circ}$ means 
$\bar{b}\bar{a}\bar{b}= \bar{b}\bar{a}$ in $R$. That is, 
$(1-b)(1-a)(1-b) = (1-b)(1-a)$.
That is, $1-a-b+ab-b+ba+bb-bab = 1-b-a+ba$.
That is, $bab = ab$, which is $ab\sqsubseteq b$ in $R$. 
\item[(c)] We already know $e_0$ and $e_2$ are idempotent.
They are also orthogonal: $e_0e_2 = ab(1-b)=ab-ab=0$
and $e_2e_0=(1-b)ab=ab-bab=0$. 
Therefore $e_0+e_2 = ab + 1 - b = 1-(1-a)b = \overline{e_1}$
is idempotent. So $e_1$ is idempotent and $e_0+e_1+e_2=1$.
The calculations $e_0e_1=ab(1-a)b=ab-abab=0$ and 
$e_1e_2=(1-a)b(1-b)=b-ab-b+abb=0$ complete the proof.
\end{enumerate}
\end{proof}

We can extend part (c) inductively to obtain: 

\begin{proposition}\label{weakcommute}
Suppose $a_1, \dots , a_n$ are idempotents such that $a_ia_j\sqsubseteq a_j$ for $i\le j$.
Then $e_i = \overline{a_i}a_{i+1}a_{i+2} \dots a_{n}$ for $0\le i \le n$
(in particular, $e_0 = a_1a_2\dots a_n$ and $e_n = \overline{a_n}$)
defines a complete list of orthogonal idempotents. 
\end{proposition}

Suppose $\CX$ is an additive category in which idempotents split.
Our results apply to the endomorphism monoid $\CX (A,A)$ of
each object $A\in \CX$. If $a$ is an idempotent on $A$, we have
a splitting:
\begin{equation*}
\xymatrix{
A \ar[rr]^-{a} \ar[d]_-{r_a} && A \ar[d]^-{r_a} \\
aA \ar[rr]_-{1_{aA}} \ar[rru]_-{i_{a}} && aA \ .}
\end{equation*} 
Yet, we also have a splitting for $\bar{a} = 1-a$ which incidentally provides a kernel
$\bar{a}A$ for $a$ and so a direct sum decomposition of $A$:
$$A \cong \bar{a}A\oplus aA \ .$$

More generally, for any complete list $e_0, e_1, \dots e_n$ of orthogonal idempotents
in $\CX (A,A)$, we obtain a direct sum decomposition
$$A\cong e_0A \oplus e_1A \oplus \dots \oplus e_nA \ .$$

\end{document}